\documentclass[11pt]{article}

\usepackage{amsmath}
\usepackage{amssymb}
\usepackage{amsfonts}
\usepackage{amsthm}
\usepackage{array}
\usepackage{bm}
\usepackage[shortlabels]{enumitem}
\usepackage{stmaryrd}

\usepackage{amsxtra}
\usepackage{array}
\usepackage{mathrsfs}
\usepackage{slashed}
\usepackage{xcolor}
\usepackage{tikz-cd}
\usepackage{tkz-euclide} 
\usepackage{pgfplots}

\newcolumntype{C}[1]{>{\centering\hspace{0pt}}p{#1}}
\usepackage[margin=1in]{geometry}

\newcommand{\Spin}{\mathrm{Spin}}

\newcommand{\SU}{\mathrm{SU}}
\newcommand{\Sp}{\mathrm{Sp}}

\newcommand{\G}{\mathrm{G}}

\newcommand{\Spinor}{\slashed{S}}

\newcommand{\Z}{\mathbb{Z}}

\newcommand{\R}{\mathbb{R}}
\newcommand{\C}{\mathbb{C}}

\newcommand{\CP}{\mathbb{CP}}

\newcommand{\vol}{\mathrm{vol}}

\newtheorem{thm}{Theorem}[section]
\newtheorem{prop}[thm]{Proposition}
\newtheorem{lem}[thm]{Lemma}
\newtheorem{cor}[thm]{Corollary}

\theoremstyle{definition}
\newtheorem{defn}[thm]{Definition}
\newtheorem{example}[thm]{Example}
\newtheorem{rmk}[thm]{Remark}

\newtheorem{convention}[thm]{Convention}

\usepackage[nottoc]{tocbibind}
\numberwithin{equation}{section}

\usepackage{hyperref}
\hypersetup{colorlinks, linkcolor=blue, citecolor=blue}

\title{Hyperbolicity and Schwarz Lemmas \\
in Calibrated Geometry}
\author{Kyle Broder, Anton Iliashenko, Jesse Madnick}
\date{December 2025}

\newcommand{\Addresses}
{{  \bigskip
\noindent	\textsc{The University of Queensland} \par\nopagebreak
\noindent	\textsc{Brisbane, Australia} \par\nopagebreak
\noindent	\texttt{k.broder@uq.edu.au} \\

\medskip
\noindent	\textsc{Beijing Institute of Mathematical Sciences and Applications} \par\nopagebreak
\noindent	\textsc{Beijing, China} \par\nopagebreak
\noindent	\texttt{antoniliashenko@bimsa.cn} \\

\medskip
\noindent	\textsc{Seton Hall University} \par\nopagebreak
\noindent	\textsc{South Orange, NJ, United States} \par\nopagebreak
\noindent	\texttt{jesse.ochs.madnick@gmail.com} \\
}}

\begin{document}

\maketitle

	\begin{abstract} This paper has two main objectives.  First, for an arbitrary calibrated manifold $(X,\phi)$, we define notions of $R_\phi$-hyperbolicity and $\phi$-hyperbolicity, which respectively generalize the notions of Kobayashi and Brody hyperbolicity from complex geometry.  To make sense of the former, we introduce the ``KR $\phi$-metric," a decreasing Finsler pseudo-metric that specializes to the Kobayashi-Royden pseudo-metric in the K\"{a}hler case.  We prove that $R_\phi$-hyperbolicity implies $\phi$-hyperbolicity, and give examples showing that the converse fails in general.  Moreover, for constant-coefficient, inner M\"{o}bius rigid calibrations $\phi$ in $\R^n$, we completely characterize those domains that are $\phi$-hyperbolic.

Second, we derive a Schwarz lemma for Smith immersions (a.k.a. conformal $\phi$-curves) into an arbitrary calibrated manifold $(X, \phi)$, thereby extending the Schwarz lemma for holomorphic curves into K\"{a}hler manifolds.  The relevant Bochner formula features the ``$\phi$-sectional curvature," a new notion that includes both the scalar and holomorphic sectional curvatures as special cases.  As an application, we prove that calibrated geometries with $\phi$-sectional curvature bounded above by a negative constant are $R_\phi$-hyperbolic, generalizing the corresponding result from complex geometry.  As another application, we calculate the KR $\phi$-metric of real, complex, and quaternionic hyperbolic spaces equipped with their natural calibrations.
\end{abstract}

  \tableofcontents

\pagebreak

\section{Introduction}

\indent \indent For a K\"{a}hler manifold $(X, g_X, \omega)$, hyperbolicity refers to the interplay of three \emph{a priori} unrelated concepts: the scarcity of holomorphic maps $\C \to X$, the non-degeneracy of invariant pseudo-distances, and strongly negative curvature.  More precisely, we say that:
\begin{enumerate}[(1)]
\item $X$ is \emph{Brody hyperbolic} if every holomorphic map $\C \to X$ is constant.
\item $X$ is \emph{Kobayashi hyperbolic} if the Kobayashi pseudo-distance is non-degenerate (see (\ref{eq:KobDist})).
\item $X$ has \emph{strongly negative holomorphic sectional curvature} if its holomorphic sectional curvature is bounded above by a negative constant.
\end{enumerate}
The implications (3) $\implies$ (2) $\implies$ (1) are well-known.  When $X$ is compact, Brody \cite{brody1978compact} proved that (1) and (2) are equivalent. 

The primary aim of this paper is to extend all three notions to arbitrary calibrated manifolds $(X, g_X, \phi)$, and to generalize the implications (3) $\implies$ (2) $\implies$ (1).  By a ``calibrated manifold,” we mean an oriented Riemannian $n$-manifold $(X, g_X)$ equipped with a comass one, closed $k$-form $\phi \in \Omega^k(X)$ called a calibration.

Extending (1) and (2) requires a suitable class of maps between calibrated manifolds.  Such a class is provided by Smith immersions.  A \emph{Smith immersion} (or \emph{conformal $\phi$-curve}) is a smooth map $f \colon (\Sigma^k, g_\Sigma, \vol_\Sigma) \to (X^n, g_X, \phi)$ satisfying
\begin{align*}
f^*g_X & = \lambda^2 g_\Sigma \\
f^*\phi & = \lambda^k \vol_\Sigma
\end{align*}
for some function $\lambda \colon \Sigma \to [0,\infty)$ called the \emph{conformal factor}.  The first equation shows that Smith immersions are weakly conformal maps.  Consequently, if $f$ is a Smith immersion, then at every $x \in \Sigma$, we have either $\mathrm{rank}(df_x) = k$ or $\mathrm{rank}(df_x) = 0$.  (Thus, in spite of the terminology, Smith immersions need not be immersions.  A Smith immersion is called \emph{strict} if it is an immersion.)  When $(X, g_X, \omega)$ is a K\"{a}hler manifold, the Smith immersions $f \colon (\Sigma^2, g_\Sigma, \vol_\Sigma) \to (X, g_X, \omega)$ are precisely the holomorphic curves into $X$.

Smith immersions can be viewed as the ``conformal mapping analog” of calibrated submanifolds.  Indeed, the image of a strict Smith immersion is an immersed $\phi$-calibrated submanifold of $X$.  Conversely, every immersed $\phi$-calibrated submanifold can be parametrized by a Smith immersion.  More interestingly, just as calibrated submanifolds are homologically volume minimizing, it turns out that Smith immersions are homologically $k$-energy minimizing.  We review the basics of Smith immersions in $\S$\ref{sec:ConformallyCalibrating}.

We say that $(X, g_X, \phi)$ is \emph{$\phi$-hyperbolic} if every Smith immersion $f \colon (\R^k, g_0, \vol_0) \to (X, g_X, \phi)$ is constant, where $g_0$ is the flat metric.  When $\phi = \omega$ is a K\"{a}hler form, $\omega$-hyperbolicity is precisely Brody hyperbolicity.  When $\phi = \vol_X$ is a volume form, $\vol_X$-hyperbolicity is related to the failure of 1-quasiregular ellipticity \cite{bonk2001quasiregular}.  For generic constant-coefficient calibrations on $\R^n$, we classify the $\phi$-hyperbolic domains:

\begin{thm} \label{thm:Main1} Let $\phi \in \Lambda^k(\R^n)^*$ be a constant-coefficient, inner M\"{o}bius rigid calibration, $k \geq 3$.  Let $U \subset \R^n$ be a domain, and equip $U$ with the flat metric $g_0$.  Then $(U, g_0)$ is $\phi$-hyperbolic if and only if $U$ contains no affine $\phi$-planes. 
\end{thm}

A result of \cite{ikonen2024liouville} implies that the generic constant-coefficient calibration on $\R^n$ is inner M\"{o}bius rigid.  Moreover, the normalized powers of the standard K\"{a}hler form $\frac{1}{p!}\omega^p \in \Lambda^{2p}(\R^{2n})^*$ for $p \geq 2$, as well as every codimension 2 calibration, is inner M\"{o}bius rigid \cite{ikonen2024liouville}.


To extend (3), we need to generalize the notion of holomorphic sectional curvature.  To this end, when $(X, g_X, \phi)$ is a calibrated manifold, we say that the \emph{$\phi$-sectional curvature} of a $\phi$-plane $\xi \subset TX$ is the sum of sectional curvatures of $2$-planes in $\xi$:
$$\mathrm{Sec}^\phi(\xi) := \sum_{1 \leq i,j \leq k} \mathrm{Sec}(e_i \wedge e_j),$$
where $\{e_1, \ldots, e_k\}$ is a $g_X$-orthonormal basis of $\xi$.  When $\phi = \omega$ or $\phi = \vol_X$, the $\phi$-sectional curvature is simply the holomorphic sectional or scalar curvature, respectively.  With this notion in hand, we generalize the implication (3) $\implies$ (1):

\begin{thm} \label{thm:Main2} Let $(X, g_X, \phi)$ be a calibrated manifold.  If $\mathrm{Sec}^\phi \leq -A$ for some $A > 0$, then $(X, g_X)$ is $\phi$-hyperbolic.
\end{thm}

To prove Theorem \ref{thm:Main2}, we will establish a Schwarz lemma for Smith immersions.  Schwarz lemmas have long been a powerful tool in complex geometry.  Recall that the classic statement asserts that every holomorphic function $f \colon \mathbb{D} \to \mathbb{D}$ on the unit disk $\mathbb{D} = \{z \in \C \colon |z| < 1\}$ with $f(0) = 0$ satisfies $|f(z)| \leq |z|$.  Pick reinterpreted this as saying that every holomorphic map $f \colon \mathbb{D} \to \mathbb{D}$ is distance non-increasing (with respect to the Poincar\'{e} distance), and Ahlfors generalized this distance non-increasing property to holomorphic maps $f \colon \mathbb{D} \to \Sigma$, where $\Sigma$ is a Riemann surface with Gauss curvature $K_\Sigma \leq -4$.  Ahlfors' Schwarz Lemma is a special case of the following result:

\begin{thm}[Schwarz Lemma] \label{thm:Main3} Let $f \colon (\Sigma^k, g_\Sigma, \vol_\Sigma) \to (X^n, g_X, \phi)$ be a smooth map, where $\phi \in \Omega^k(X)$ is a calibration.  Suppose that $\Sigma$ is complete and has Ricci curvature bounded below.
\begin{enumerate}[(a)]
\item Suppose $f$ is weakly conformal and $k$-harmonic.  If $\mathrm{Scal}_\Sigma \geq -S$ for $S \geq 0$, and if $\mathrm{Sec}_X \leq -B$ for $B > 0$, then $|df|^2 \leq \frac{S}{(k-1)B}$.  Consequently, $\mathrm{dist}_X(f(p), f(q)) \leq \sqrt{ \frac{S}{k(k-1)B} }\,\mathrm{dist}_\Sigma(p,q)$ for all $p,q \in \Sigma$.
\item Suppose $f$ is a Smith immersion.  If $\mathrm{Scal}_\Sigma \geq -S$ for $S \geq 0$, and if $\mathrm{Sec}^\phi_X \leq -A$ for $A > 0$, then $|df|^2 \leq \frac{kS}{A}$.   Consequently, $\mathrm{dist}_X(f(p), f(q)) \leq \sqrt{ \frac{S}{A}}\,\mathrm{dist}_\Sigma(p,q)$ for all $p,q \in \Sigma$.
\end{enumerate}
\end{thm}

Finally, we indicate how to generalize (2).  On a complex manifold $X$, the \emph{Kobayashi-Royden pseudo-metric} is the Finsler pseudo-metric $F_X \colon TX \to [0, \infty]$ defined by
\begin{equation} \label{eq:KRPM}
F_X(v_p) = \inf\left\{ \frac{1}{r} \colon \exists \text{ holomorphic }f \colon \mathbb{D}(r) \to X \text{ s.t. }f(0) = p, f'(0) = v_p \right\}\!,
\end{equation}
where $\mathbb{D}(r) = \{z \in \C \colon |z| < r\}$.  Intuitively, $F_X(v)$ is the reciprocal of the radius of the largest holomorphic disk in $X$ that contains $v$ as a tangent vector.  Notice that a lower bound on $F_X$ is essentially equivalent to a Schwarz lemma for holomorphic maps $\mathbb{D} \to X$.  A remarkable feature of $F_X$ is its ``decreasing property”: for every holomorphic map $f \colon X \to Y$, we have $f^*F_Y \leq F_X$.  This implies in particular that $F_X$ is invariant under biholomorphisms.  Regarding regularity, Royden proved \cite{royden2006remarks} that $F_X$ is upper semicontinuous, but $F_X$ need not be continuous in general \cite[Example 3.5.37]{kobayashi2013hyperbolic}.

\indent Integrating the Kobayashi-Royden pseudo-metric results in the \emph{Kobayashi pseudo-distance}
\begin{equation} \label{eq:KobDist}
d_X^{\mathrm{Kob}}(p,q) = \inf_\gamma \int_0^1 F_X(\gamma'(t))\,dt,
\end{equation}
where the infimum is taken over all piecewise-smooth curves $\gamma \colon [0,1] \to X$ from $p$ to $q$.  Royden proved that $d_X^{\mathrm{Kob}}$ is a genuine distance function (i.e., $p \neq q$ implies $d_X^{\mathrm{Kob}}(p,q) > 0$) if and only if every $p \in X$ has a neighborhood $U \subset X$ and a constant $c > 0$ such that $F_X(v) \geq c |v|$ for all $v \in TU$.  Such complex manifolds $X$ are said to be \emph{Kobayashi hyperbolic}.  It is well-known that Kobayashi hyperbolicity implies Brody hyperbolicity.

In general, for an arbitrary calibrated manifold $(X, g_X, \phi)$, we define its \emph{KR $\phi$-metric} to be the function $\mathcal{K}_{(X,\phi)} \colon TX \to [0,\infty]$ given by
$$\mathcal{K}_{(X, \phi)}(v_p) = \inf\!\left\{ a > 0 \colon \exists f \in \mathrm{SmIm}(B^k, X), w \in T_0B^k, |w| = 1 \text{ s.t. } f(0) = p,\, df_0(w) = \frac{1}{a}v_p \right\}\!,$$
where $\mathrm{SmIm}(B^k, X)$ denotes the collection of Smith immersions $f \colon (B^k, g_1, \vol_1) \to (X, g_X, \phi)$, with $g_1$ denoting the Poincar\'{e} metric and $\vol_1$ its volume form.  When $\phi = \omega$ is a K\"{a}hler calibration, $\mathcal{K}_{(X,\omega)}$ reduces to the Kobayashi-Royden pseudo-metric.  We prove in $\S$\ref{sub:KR} that $\mathcal{K}_{(X,\phi)}$ satisfies a generalization of the decreasing property, and is therefore invariant under Smith equivalences (see $\S$\ref{sec:ConformallyCalibrating} for the definition).  In $\S$\ref{sub:Examples-KR}, we calculate the KR $\phi$-metric of real, complex, and quaternionic hyperbolic spaces for suitable calibrations $\phi$. \\
\indent We shall say that $(X, g_X)$ is \emph{$R_\phi$-hyperbolic} if Royden's criterion holds: every point in $X$ has a neighborhood $U \subset X$ and a constant $c > 0$ such that $\mathcal{K}_{(X,\phi)}(v) \geq c|v|_X$ for all $v \in TU$.  Under the assumption that $\mathcal{K}_{(X,\phi)}$ is upper semicontinuous, we say that $(X,g_X)$ is \emph{$K_\phi$-hyperbolic} if the pseudo-distance $d_{(X,\phi)}(p,q) := \inf_\gamma \int_0^1 \mathcal{K}_{(X,\phi)}(\gamma'(t))\,dt$ is non-degenerate.  The desired generalization of (3) $\implies$ (2) $\implies$ (1) is as follows:

\begin{thm}  \label{thm:Main4} Let $(X, g_X, \phi)$ be a calibrated manifold.  Then
$$\mathrm{Sec}^\phi \leq -A \text{ for some } A > 0 \ \implies \ X \text{ is }R_\phi\text{-hyperbolic} \ \implies \ X \text{ is }\phi\text{-hyperbolic.}$$
Moreover, if $\mathcal{K}_{(X,\phi)}$ is upper semicontinuous, then
$$X\text{ is }R_\phi\text{-hyperbolic} \ \implies \ X \text{ is }K_\phi\text{-hyperbolic} \ \implies X \text{ is }\phi\text{-hyperbolic.}$$
\end{thm}

We also show that the converse of the implication $R_\phi$-hyperbolic $\implies$ $\phi$-hyperbolic is false in general.

\begin{thm} \label{thm:Main5} Let $\phi \in \Lambda^k(\R^n)^*$ be an inner M\"{o}bius rigid, constant-coefficient, elliptic calibration, $k \geq 2$.  Then there exists a domain $U \subset \R^n$ that is $\phi$-hyperbolic but not $R_\phi$-hyperbolic.  Moreover, if $\mathcal{K}_{(U,\phi)}$ is upper semicontinuous, then $U$ is not $K_\phi$-hyperbolic.
\end{thm}

\indent Finally, we remark that analogs of Brody and Kobayashi hyperbolicity have recently been discovered in a purely Riemannian context.  By using conformal harmonic maps in place of holomorphic ones, Forstneri\v{c}-Kalaj \cite{forstnerivc2024schwarz}, Drinovec Drnov\v{s}ek-Forstneri\v{c} \cite{drnovsek2021hyperbolic}, and Forstneri\v{c} \cite{forstnerivc2023domains} have proposed and studied various hyperbolicity notions for domains in euclidean space, and Gaussier-Sukhov \cite{gaussier2025kobayashi}, \cite{gaussier2024kobayashi} have extended them to arbitrary Riemannian manifolds.  These papers were a primary inspiration for this work.  The following table provides a dictionary between various hyperbolicity notions encountered in the literature.
$$\begin{tabular}{| l | l | l |} \hline
Complex geometry & Calibrated geometry & Riemannian geometry \\ \hline \hline
Holomorphic maps & Smith immersions & Conformal harmonic maps \\ \hline
Brody hyperbolic & $\phi$-hyperbolic & Weakly hyperbolic \\ \hline
Kobayashi hyperbolic & $K_\phi$-hyperbolic & Kobayashi hyperbolic \\ \hline
Kobayashi hyperbolic & $R_\phi$-hyperbolic & $R$-hyperbolic \\ \hline
\end{tabular}$$

\subsection{Notation and conventions}

We adopt the following notation and conventions:
\begin{itemize}
\item By ``manifold," we always mean a connected smooth manifold, unless indicated otherwise.
\item We let $B^n(R) = \{x \in \R^n \colon |x| < R\}$ denote the open $n$-ball of radius $R > 0$, and write $B^n := B^n(1)$.  On $B^n(R)$, we let $g_R$ denote the Poincar\'{e} metric, a complete metric of constant curvature $-\frac{4}{R^2}$, which we briefly recall in $\S$\ref{sub:Hyperbolic-Space}.  We let $| \cdot |_R = \sqrt{g_R( \cdot, \cdot )}$ denote the corresponding norm, and let $\vol_R \in \Omega^n(B^n)$ denote the volume form.
\item Let $U \subset \R^n$ be an open set.  On $U$, we let $g_0$ denote the flat metric, let $| \cdot |_0 = \sqrt{g_0(\cdot, \cdot)}$ denote the euclidean norm, and let $\vol_0 = dx^1 \wedge \cdots \wedge dx^n \in \Omega^n(U)$ denote the volume form.
\item Let $(X, g_X)$ be a Riemannian $n$-manifold.  We let $\mathrm{dist}_X  \colon X \times X \to \R$ denote the corresponding distance function.  When $X$ is oriented, we let $\vol_X \in \Omega^n(X)$ denote the volume form.
\item For a linear operator $A \colon V \to W$ between finite-dimensional inner product spaces, we denote its Hilbert-Schmidt and operator norm by, respectively
\begin{align*}
|A| & = |A|_{\mathrm{HS}} = \sqrt{ \sum_{i=1}^n |Ae_i|^2 }, & \Vert A \Vert & = \Vert A \Vert_{\mathrm{op}} = \sup_{|v| = 1} \frac{|Av|}{|v|},
\end{align*}
where $\{e_1, \ldots, e_n\}$ is an orthonormal basis of $V$.
\item Given a Riemannian manifold $(X,g)$, our sign convention for its Riemann curvature tensor is $R(U,V,Z,W) = g(\nabla_U \nabla_V Z - \nabla_V \nabla_U Z - \nabla_{[U,V]} Z , W)$.  In a local orthonormal frame $\{e_1, \ldots, e_n\}$, we write $R_{ijk\ell} := R(e_i, e_j, e_k, e_\ell)$.  The sectional curvature on a unit $2$-vector is denoted $\mathrm{Sec}_X(e_i \wedge e_j) = R_{ijji}$ (no sum).  In the usual way, the Ricci and scalar curvatures of $(X,g)$ in a local frame are denoted $R_{ij} := R_{kijk}$ and $\mathrm{Scal}_X := R_{ii} = R_{kiik}$, where here we use the Einstein summation convention.
\item For a smooth map between Riemannian manifolds $f:(X,g_X) \to (Y,g_Y)$, we denote its \emph{$k$-tension}, a section of $f^*TY$, by
\begin{align*}
\tau_k(f)  := \mathrm{div}(|df|^{k-2} df) =\mathrm{tr}_g (\nabla (|df|^{k-2} df)),
\end{align*}
where $\nabla$ is the induced connection on $T^{*}X \otimes f^{*}TY$.  Note that although $f$ is smooth, $\tau_k(f)$ need not be smooth for $k$ odd.  We say that $f$ is \emph{$k$-harmonic} if $\tau_k(f) = 0$.  The operator $\tau_k$ is (overdetermined) elliptic, and is non-degenerate for $k = 2$, but is degenerate at critical points of $f$ for $k > 2$.  When $Y = \R$, we refer to $\Delta_kf := \tau_k(f)$ as the \emph{$k$-Laplacian}, and say that $f$ is \emph{$k$-subharmonic} if $\Delta_k f \geq 0$.
\end{itemize}

\noindent \textbf{Acknowledgements:} We thank Gavin Ball, Stepan Hudecek, Toni Ikonen, Jason Lotay, Pekka Pankka, and Tommaso Pacini for valuable conversations.  We especially thank Da Rong Cheng and Spiro Karigiannis for their careful readings of an earlier version of this manuscript.  The third author thanks the American Institute of Mathematics.
 
\section{Preliminaries}

\indent \indent We begin by setting foundations.  Section \ref{sec:Calibrated} rapidly recalls the basics of calibrated geometry.  Section \ref{sec:ConformallyCalibrating} concerns \emph{conformally calibrating maps} and \emph{Smith equivalences}, which are our preferred notions of ``morphism" and ``isomorphism" between calibrated manifolds having calibrations of the same degree.  Special attention is paid to \emph{Smith immersions} (also known as \emph{conformal $\phi$-curves}), which provide a conformal mapping approach to the study of calibrated submanifolds.  Indeed, where calibrated submanifolds are minimal (and compact ones are homologically volume minimizing), we will recall that Smith immersions are $k$-harmonic (and those from compact domains are homologically $k$-energy minimizing).  Finally, in $\S$\ref{sub:Hyperbolic-Space}, we quickly review some facts about the Poincar\'{e} ball model of hyperbolic space that we will need in this work.

\subsection{Calibrated manifolds} \label{sec:Calibrated}

\begin{defn}[\cite{harvey1982calibrated}] Let $(X, g_X)$ be an oriented Riemannian $n$-manifold.
\begin{itemize}
\item A \emph{semi-calibration} on $X$ is a $k$-form $\phi \in \Omega^k(X)$ that has co-mass one.
\item A \emph{calibration} on $X$ is a $k$-form $\phi \in \Omega^k(X)$ that has co-mass one and is closed ($d\phi = 0$).
\end{itemize}
A \emph{(semi-)calibrated manifold} $(X, g_X, \phi)$ is an oriented Riemannian manifold $(X, g_X)$ equipped with a (semi-)calibration.
\end{defn}

\begin{defn}[\cite{harvey1982calibrated}] Let $(X, g_X, \phi)$ be a semi-calibrated manifold, where $\phi \in \Omega^k(X)$.
\begin{itemize}
\item A \emph{$\phi$-calibrated plane} (or simply a \emph{$\phi$-plane}) is an oriented $k$-plane $E \subset T_xX$ satisfying $\phi|_E = \vol_E$, where $\vol_E \in \Lambda^k(E^*)$ is the volume form of $E$.  We denote the Grassmannian of $\phi$-planes by
$$\mathrm{Gr}(\phi) := \{E \in \mathrm{Gr}_k(TX) \colon E \text{ is a }\phi\text{-plane}\} \subset \mathrm{Gr}_k^+(TX).$$
\item A \emph{$\phi$-calibrated submanifold} (or simply a \emph{$\phi$-submanifold}) is an oriented $k$-dimensional submanifold $\Sigma \subset X$ for which each of its tangent planes $T_x\Sigma \subset T_xX$ is $\phi$-calibrated.
\end{itemize}
\end{defn}

\begin{thm}[Fundamental theorem of calibrations \cite{harvey1982calibrated}] \label{thm:FundThmCal} Let $(X, g_X, \phi)$ be a calibrated manifold.  If $\Sigma \subset X$ is $\phi$-calibrated, then $\Sigma$ is minimal.  Moreover, if $\Sigma$ is compact and without boundary, then $\Sigma$ is volume-minimizing in its homology class.
\end{thm}

\indent For certain applications, it will be convenient to impose the following non-degeneracy assumption.

\begin{defn}[\cite{harvey2009introduction}] A semi-calibration $\phi \in \Omega^k(X)$ is called \emph{elliptic} if for every $x \in X$ and every $v \in T_xX$, there exists a $\phi$-plane $\xi \in \mathrm{Gr}(\phi)|_x$ such that $v \in \xi$.
\end{defn}

As we now recall, all the calibrations arising in connection with special holonomy are elliptic.

\begin{example}[\cite{harvey1982calibrated}] Let $(X, g_X)$ be an oriented Riemannian $n$-manifold.
\begin{enumerate}
\item The volume form $\vol_X \in \Omega^n(X)$ is an elliptic calibration.  This seemingly trivial edge case will play an important role in this work.
\item Suppose $n = 2m$.  If $X$ is K\"{a}hler, then a K\"{a}hler form $\omega \in \Omega^2(X)$ is an elliptic calibration.  More generally, its normalized powers $\frac{1}{p!}\omega^p \in \Omega^{2p}(X)$ are all elliptic calibrations.
\item Suppose $n = 2m$.  If $X$ is Calabi-Yau, then a holomorphic volume form $\Upsilon \in \Omega^{n,0}(X)$ defines a one-parameter family of elliptic calibrations $\mathrm{Re}(e^{-i\theta}\Upsilon) \in \Omega^n(X)$, where $\theta \in [0,2\pi)$ is a constant.  These are known as the \emph{special Lagrangian calibrations}.
\item Suppose $n = 4m$.  If $X$ is quaternionic-K\"{a}hler, then its QK $4$-form $\Psi \in \Omega^4(X)$ is an elliptic calibration.  Locally, we may write $\Psi = \frac{1}{6}(\omega_1^2 + \omega_2^2 + \omega_3^2)$, where $(\omega_1, \omega_2, \omega_3)$ is a triple of (locally defined) non-degenerate $2$-forms associated to an admissible frame $(J_1, J_2, J_3)$.
\item Suppose $n = 7$.  A $\mathrm{G}_2$-structure on $X$ defines elliptic semi-calibrations $\varphi \in \Omega^3(X)$ and $\ast \varphi \in \Omega^4(X)$.  The $\mathrm{G}_2$-structure is called \emph{closed} (resp., \emph{co-closed}) if $\varphi$ (resp., $\ast \varphi$) are calibrations.
\item Suppose $n = 8$.  A $\Spin(7)$-structure on $X$ defines an elliptic semi-calibration $\Phi \in \Omega^4(X)$.  The $\Spin(7)$-structure is called \emph{torsion-free} if $\Phi$ is closed.
\end{enumerate}
\end{example}
\noindent We remark also that hyperk\"{a}hler $4m$-manifolds admit a large supply of calibrations: see \cite{bryant1989submanifolds}.

\subsection{Conformally calibrating maps} \label{sec:ConformallyCalibrating}

\indent \indent We now define a natural notion of ``map" between two calibrated manifolds having calibrations of the same degree.  For this, it will be convenient to first discuss the wider class of weakly conformal maps. \\
\indent A map $f \colon (X^m, g_X) \to (Y^n, g_Y)$ between Riemannian manifolds is \emph{weakly conformal} if there exists a non-negative function $\lambda \colon X \to [0,\infty)$ such that $f^*g_Y = \lambda^2 g_X$.  Note that if $f$ is weakly conformal, then for each $x \in X$, the conformal factor $\lambda(x)$ is given by
\begin{equation} \label{eq:Lambda}
\lambda(x) = \frac{1}{\sqrt{m}}|df_x| = \Vert df_x \Vert.
\end{equation}
We say $f$ is \emph{conformal} if $\lambda > 0$, and an \emph{isometry} if $\lambda = 1$. \\
\indent If $f$ is weakly conformal, then for each point $x \in X$, we have either $\mathrm{rank}(df_x) = m$ (if $\lambda(x) > 0$) or $\mathrm{rank}(df_x) = 0$ (if $\lambda(x) = 0$).  In particular, if $m > n$, then every weakly conformal map is locally constant.  To avoid this trivial case, we shall always assume that $m \leq n$.  Note also that conformal maps are necessarily immersions.

\begin{defn}[\cite{cheng2021bubble}] Let $(X^m, g_X, \alpha)$ and $(Y^n, g_Y, \beta)$ be calibrated manifolds such that $k := \deg(\alpha) = \deg(\beta)$ and $m \leq n$.  A smooth map $f \colon X \to Y$ is \emph{conformally calibrating} if there exists a non-negative function $\lambda \colon X \to [0,\infty)$ such that
\begin{align*}
f^*g_Y & = \lambda^2 g_X \\
f^*\beta & = \lambda^k \alpha.
\end{align*}
Note that constant maps are conformally calibrating.
\end{defn}

\begin{example} Let $f \colon (X^{2m}, g_X, \omega_X) \to (Y^{2n}, g_Y, \omega_Y)$ be a map between two K\"{a}hler manifolds.  One can prove that such an $f \colon X \to Y$ is conformally calibrating if and only if $f$ is weakly conformal and holomorphic.  For $m \geq 2$, this is equivalent to requiring that $f$ be either a constant map, or else a conformal holomorphic immersion with $\lambda = \text{constant}$.
\end{example}

\indent We shall provide more examples of conformally calibrating maps later in this section.  For now, we observe the following basic fact.

\begin{prop} \label{prop:Composition} Let $F \colon (X^m, g_X, \alpha) \to (Y^n, g_Y, \beta)$ and $G \colon (Y^n, g_Y, \beta) \to (Z^p, g_Z, \gamma)$ be smooth maps.  If $F$ and $G$ are conformally calibrating, then so is $G \circ F$.
\end{prop}

\begin{proof} Since $F$ is conformally calibrating, there exists $\lambda \colon X \to [0,\infty)$ such that $F^*g_Y = \lambda^2g_X$ and $F^*\beta = \lambda^k \alpha$.  Since $G$ is conformally calibrating, there exists $\mu \colon Y \to [0,\infty)$ such that $G^*g_Z = \mu^2 g_Y$ and $G^*\gamma = \mu^k \beta$.  Setting $\nu := \lambda \cdot (\mu \circ F)$, we have
\begin{align*}
(G \circ F)^*g_Z & = F^*(G^*g_Z) = F^*(\mu^2 g_Y) = F^*(\mu^2)\,\lambda^2 g_X = \nu^2 g_X \\
(G \circ F)^*\gamma & = F^*(G^*\gamma) = F^*(\mu^k \beta) = F^*(\mu^k)\,\lambda^k \alpha = \nu^k \alpha,
\end{align*}
so $G \circ F$ is conformally calibrating.
\end{proof}

\subsubsection{Smith equivalences and automorphisms}

\indent \indent Using conformally calibrating maps, we may define the following notion of ``isomorphism" between calibrated manifolds.

\begin{defn} Let $(X^n, g_X, \alpha)$ and $(Y^n, g_Y, \beta)$ be calibrated $n$-manifolds such that $k = \deg(\alpha) = \deg(\beta)$.  A \emph{(local) Smith equivalence} is a conformally calibrating map $F \colon (X^n, g_X, \alpha) \to (Y^n, g_Y, \beta)$ that is a (resp. local) diffeomorphism.
\end{defn}

\begin{example} ${}$
\begin{enumerate}
\item Let $(\Sigma^k, g_\Sigma, \vol_{\Sigma})$ and $(\widetilde{\Sigma}^k, g_{\widetilde{\Sigma}}, \vol_{\widetilde{\Sigma}})$ be calibrated $k$-manifolds equipped with degree $k$ calibrations (i.e., both calibrations are volume forms).  Then
\begin{align*}
& F\colon (\Sigma, g_\Sigma, \vol_\Sigma) \to (\widetilde{\Sigma}, g_{\widetilde{\Sigma}}, \vol_{\widetilde{\Sigma}}) \text{ is a Smith equivalence} \\
& \ \ \ \ \ \ \ \iff \ F \text{ is an orientation-preserving conformal diffeomorphism.}
\end{align*}
\item Let $(X^{2m}, g_X, \omega_X)$ and $(Y^{2m}, g_Y, \omega_Y)$ be K\"{a}hler $2m$-manifolds.  Then
\begin{align*}
& F\colon (X^{2m}, g_X, \omega_X) \to (Y^{2m}, g_Y, \omega_Y) \text{ is a Smith equivalence} \\
& \ \ \ \ \ \ \ \iff \ F \text{ is a conformal biholomorphism.}
\end{align*}
Note that when $m = 1$, biholomorphisms are automatically conformal.
\end{enumerate}
\end{example}

\begin{example} \label{example:Inclusion} The following two basic examples will be used repeatedly.  Let $(X, g_X)$ be an oriented Riemannian manifold.
\begin{enumerate}
\item The identity map $\mathrm{Id} \colon (X, g_X, \vol_{g_X}) \to (X, \lambda^2 g_X, \vol_{\lambda^2 g_X})$ is a Smith equivalence, where $\lambda \colon X \to [0,\infty)$ is a smooth function. 
\item Let $U \subset X$ be an open set, and equip $U$ with the metric and calibration obtained from $X$ by restriction.  Then the inclusion map $\iota \colon (U, g_X, \phi) \hookrightarrow (X, g_X, \phi)$ is a local Smith equivalence.
\end{enumerate}
\end{example}

Example \ref{example:Inclusion}(b) shows in particular that the inclusion map $\iota \colon (B^n, g_0, \vol_0) \hookrightarrow (\R^n, g_0, \vol_0)$ is a local Smith equivalence.  However, we emphasize that the calibrated manifolds $(B^n, g_0, \vol_0)$ and $(\R^n, g_0, \vol_0)$, despite being diffeomorphic, are \emph{not} globally Smith equivalent.  Before moving on, we explain how discrete quotients can be viewed as local Smith equivalences.

\begin{prop}[Discrete quotients] \label{prop:Discrete} Let $(X, g_X)$ be an oriented Riemannian manifold, and let $\Gamma \leq \mathrm{Isom}(X, g_X)$ be a discrete subgroup that acts freely and properly on $X$.  Let $\pi \colon X \to Y := X/\Gamma$ denote the projection, and equip $Y$ with the unique Riemannian metric $g_Y$ such that $\pi \colon (X, g_X) \to (Y, g_Y)$ is a local isometry.  \\
\indent Let $\phi_X \in \Omega^k(X)$ be a $\Gamma$-invariant calibration on $(X, g_X)$.  Then $\phi_X$ descends to a $k$-form $\phi_Y \in \Omega^k(Y)$ that is a calibration with respect to $g_Y$.  Moreover, $\pi \colon (X, g_X, \phi_X) \to (Y, g_Y, \phi_Y)$ is a local Smith equivalence.
\end{prop}

\begin{proof} Let $\phi_X \in \Omega^k(X)$ be a $\Gamma$-invariant calibration on $(X, g_X)$.  Since $\phi_X$ is $\Gamma$-invariant, there exists a unique $\phi_Y \in \Omega^k(Y)$ such that $\phi_X = \pi^*\phi_Y$. \\
\indent We claim that $\phi_Y$ is a calibration.  First, since $0 = d\phi_X = \pi^*(d\phi_Y)$ and $\pi^*$ is injective, we see that $\phi_Y$ is closed.  Next, let $\{e_1, \ldots, e_k\}$ be a local $g_Y$-orthonormal frame on an open set $U \subset Y$.  Since $\pi$ is a Riemannian covering, by shrinking $U$ if necessary, there exists a local $g_X$-orthonormal frame $\{f_1, \ldots, f_k\}$ on $X$ having $d\pi(f_i) = e_i$.  Therefore, $\phi_Y(e_1 \wedge \cdots \wedge e_k) = \phi_Y(d\pi(e_1) \wedge \cdots \wedge d\pi(e_k)) = \phi_X(f_1 \wedge \cdots \wedge f_k) \leq 1$, and thus $\phi_Y$ has comass one. \\
\indent Finally, since $\pi^*g_Y = g_X$ and $\pi^*\phi_Y = \phi_X$, we see that $\pi$ is a local Smith equivalence.
\end{proof}

\indent Having defined our notion of ``isomorphism" between calibrated manifolds, we now briefly consider the corresponding notion of ``automorphism."

\begin{defn} Let $(X^n, g_X, \alpha)$ be a calibrated manifold.
\begin{itemize}
\item A \emph{Smith automorphism of $X$} is a Smith equivalence $F \colon (X^n, g_X, \alpha) \to (X^n, g_X, \alpha)$.  By definition, that this means that $F$ is a diffeomorphism such that there exists $\lambda \colon X \to \R^+$ for which
\begin{align*}
F^*g_X & = \lambda^2 g_X \\
F^*\alpha & = \lambda^k \alpha.
\end{align*}
We let $\mathrm{SmAut}(X, g_X, \alpha)$ denote the group of Smith automorphisms.
\item  A \emph{Gray automorphism of $X$} is a Smith automorphism $F \colon (X^n, g_X, \alpha) \to (X^n, g_X, \alpha)$ that is an isometry (i.e., it is a Smith automorphism with $\lambda = 1$).  We let $\mathrm{GrAut}(X, g_X, \alpha)$ denote the group of Gray automorphisms.
\end{itemize}
\end{defn}

\indent Note that we have inclusions
$$\begin{tikzcd}
{\mathrm{SmAut}(X,g_X,\phi)} \arrow[r, hook]                 & {\mathrm{ConfAut}(X, g_X)}             \\
{\mathrm{GrAut}(X,g_X,\phi)} \arrow[u, hook] \arrow[r, hook] & {\mathrm{Isom}(X,g_X)} \arrow[u, hook]
\end{tikzcd}$$
Here, $\mathrm{ConfAut}(X,g_X)$ and $\mathrm{Isom}(X, g_X)$ are the groups of conformal automorphisms and isometries of $(X,g_X)$, respectively.

\begin{example} ${}$
\begin{enumerate}
\item Let $(X^{2m}, g_X, \omega_X)$ be a K\"{a}hler manifold.  Then the Smith automorphisms of $X$ are precisely the conformal biholomorphisms, while the Gray automorphisms of $X$ are the isometric biholomorphisms.
\item Consider flat euclidean space $(\R^n, g_0)$ with $n \geq 3$.  Liouville's theorem on conformal maps states that every conformal automorphism of $(\R^n, g_0)$ is a (restriction of a) M\"{o}bius transformation.  In particular, if $\phi \in \Omega^k(\R^n)$ is any calibration, then $\mathrm{SmAut}(\R^n, g_0, \phi)$ is a subgroup of the group of M\"{o}bius transformations.
\end{enumerate}
\end{example}

\subsubsection{Smith immersions}

\indent \indent The most important conformally calibrating maps $f \colon \Sigma \to X$ are those for which the calibration on the domain is a top-degree form.

 \begin{defn} A \emph{Smith immersion} $f \colon (\Sigma^k, g_\Sigma, \vol_\Sigma) \to (X^n, g_X, \phi)$ is a conformally calibrating map for which the calibration on the domain is the volume form $\vol_\Sigma \in \Omega^k(\Sigma)$.  (Note that this requires $\dim(\Sigma) = \deg(\vol_\Sigma) = \deg(\phi)$.) \\
 \indent Explicitly, a Smith immersion is a map $f \colon \Sigma^k \to X^n$ for which there exists a function $\lambda \colon X \to [0,\infty)$ satisfying
\begin{align*}
f^*g_X & = \lambda^2 g_\Sigma \\
f^*\phi & = \lambda^k \vol_\Sigma.
\end{align*}
This terminology is somewhat unfortunate as \emph{Smith immersions need not be immersions}.  Indeed, if $f \colon \Sigma^k \to X^n$ is a Smith immersion, then each point $p \in \Sigma$ has either $\mathrm{rank}(df_p) = k$ or $0$.  In particular, constant maps are Smith immersions. \\
\indent We shall say that a Smith immersion $f \colon \Sigma \to X$ is \emph{strict at $p \in \Sigma$} if it is an immersion at $p$.  This is equivalent to requiring that $f$ be conformal at $p$ (i.e., $\lambda(p) > 0$).
 \end{defn}
 
  \begin{prop}[Invariance]\label{prop:Invariance} Let $f \colon (\Sigma^k, g_\Sigma, \vol_\Sigma) \to (X^n, g_X, \phi)$ be a Smith immersion.
 \begin{enumerate}[(a)]
 \item If $G \colon (S, g_S, \vol_S) \to (\Sigma, g_\Sigma, \vol_\Sigma)$ is conformally calibrating, then $f \circ G \colon S \to X$ is a Smith immersion.
  \item If $G \colon (X, g_X, \phi) \to (Y, g_Y, \beta)$ is conformally calibrating, then $G \circ f \colon \Sigma \to Y$ is a Smith immersion.
  \end{enumerate}
 \end{prop}
  
  \begin{proof} This is immediate from Proposition \ref{prop:Composition}.
  \end{proof}

\indent Strict Smith immersions are exactly the conformal maps that parametrize immersed calibrated submanifolds.  To be precise:
   
 \begin{prop}[\cite{cheng2021bubble}] \label{prop:SmithParamSub} If $f \colon (\Sigma^k, g_\Sigma, \vol_\Sigma) \to (X^n, g_X, \phi)$ is a strict Smith immersion, then its image $f(\Sigma)$ is an immersed $\phi$-calibrated submanifold.  Conversely, if $\iota \colon \Sigma \to (X^n, g_X, \phi)$ is an immersed $\phi$-calibrated submanifold, then equipping $\Sigma$ with the metric $g_\Sigma = \iota^*g_X$ and volume form $\vol_\Sigma = \iota^*\phi$ makes $\iota$ into a strict Smith immersion.
 \end{prop}
 
\indent A remarkable feature of Smith immersions is their $k$-harmonicity.  Recall that for a smooth map $f \colon (M^m, g_M) \to (N^n, g_N)$ between Riemannian manifolds, its \emph{$k$-energy} $E_k(f) \in [0,\infty]$ and \emph{$k$-tension} $\tau_k(f)$, a section of $f^*TN$, are
\begin{align*}
E_k(f) & = \frac{1}{(\sqrt{k})^k} \int_M |df|^k\,\vol_M, & \tau_k(f) & = \mathrm{div}(|df|^{k-2} df).
\end{align*}
It is well-known that the critical points of the $k$-energy functional are precisely the \emph{$k$-harmonic maps}, i.e., those satisfying the second-order PDE $\tau_k(f) = 0$.  \\
\indent Now, let $f \colon (\Sigma^k, g_\Sigma, \vol_\Sigma) \to (X^n, g_X, \phi)$ be a smooth map, where $\phi \in \Omega^k(X)$ is a calibration.  Let $\vol_{f^*g_X} := \left| \frac{\partial f}{\partial x_1} \wedge \cdots \wedge \frac{\partial f}{\partial x^k}  \right| dx^1 \wedge \cdots \wedge dx^k$ in local coordinates $(x^1, \ldots, x^k)$, and set $V(f) := \int_\Sigma \vol_{f^*g_X}$.  If $f$ is an immersion, then $\vol_{f^*g_X}$ is the volume form of the induced metric $f^*g_X$.  Then the following fundamental inequality holds \cite{cheng2023variational}, \cite{ikonen2024liouville}:
\begin{equation} \label{eq:FundIneq}
\frac{1}{(\sqrt{k})^k}|df|^k \vol_\Sigma \geq \vol_{f^*g_X} \geq f^*\phi.
\end{equation}
Moreover, we have equality $\frac{1}{(\sqrt{k})^k}|df|^k \vol_\Sigma = \vol_{f^*g_X} = f^*\phi$ if and only if $f$ is Smith.  Now, if $\Sigma$ is compact, then integrating (\ref{eq:FundIneq}) gives:
$$E_k(f) \geq V(f) \geq \int_\Sigma f^*\phi.$$
This shows that if $\Sigma$ is compact and without boundary, then Smith immersions $f \colon \Sigma \to X$ are homological minimizers of the $k$-energy functional.  In fact, we have the following mapping analog of the fundamental theorem of calibrations:

\begin{thm}[Smith's theorem \cite{smith2011theory}] If $f \colon (\Sigma^k, g_\Sigma, \vol_\Sigma) \to (X^n, g_X, \phi)$ is a Smith immersion, then $f$ is $k$-harmonic.  Moreover, if $\Sigma$ is compact and without boundary, then $f$ is $k$-energy minimizing in its homology class.
\end{thm}

\begin{rmk} Note that if a Smith immersion $f \colon (\Sigma^k, g_\Sigma, \vol_\Sigma) \to (X^n, g_X, \phi)$ is isometric, then it is simultaneously $2$-harmonic and $k$-harmonic.
\end{rmk}

\indent We pause to make some historical remarks and clarify terminology.  Smith immersions were introduced in 2011 by Smith \cite{smith2011theory}, who termed them ``multiholomorphic maps."  They were then studied in \cite{cheng2021bubble} and \cite{cheng2023variational} where they were renamed ``Smith maps."  In \cite{iliashenko2023special}, a dual notion of ``Smith submersion" was proposed, and shown to be related to calibrated fibrations.  (We will not consider Smith submersions in this work.) \\
\indent Independently, Pankka \cite{pankka2020quasiregular} has defined a generalization of Smith immersions known as ``$K$-quasiregular $\phi$-curves," where $K \geq 1$ is a constant called the \emph{distortion}.  Those with $K = 1$ are called ``conformal $\phi$-curves," which is another name for Smith immersions.  The $K$-quasiregular perspective has received considerable attention: see, for example, \cite{heikkila2024quasiregularcohom}, \cite{heikkila2023quasiregular},  \cite{hitruhin2025quasiconformal}, \cite{ikonen2024quasiregular},  \cite{onninen2021quasiregular}.  In particular, Ikonen and Pankka \cite{ikonen2024liouville} proved a remarkable generalization of Liouville's theorem on conformal maps that holds for a large class of Smith immersions. \\
\indent We now mention some examples of interest.

\begin{example} ${}$
\begin{enumerate}
\item Let $(X^n, g_X, \vol_X)$ be an oriented Riemannian $n$-manifold.  Then a map $f \colon (\Sigma^n, g_\Sigma, \vol_\Sigma) \to (X^n, g_X, \vol_X)$ is Smith if and only if it is an orientation-preserving, weakly conformal map.
\item Let $(X^{2n}, g_X, \omega_X)$ be a K\"{a}hler manifold.  Then a map $f \colon (\Sigma^2, g_\Sigma, \vol_\Sigma) \to (X^{2n}, g_X, \omega_X)$ is Smith if and only if it is a holomorphic curve.
\item Let $(X^7, g_X, \varphi)$ be a $7$-manifold with a closed $\G_2$-structure $\varphi \in \Omega^3(X)$.  The Smith immersions $f \colon (\Sigma^3, g_\Sigma, \vol_\Sigma) \to (X^7, g_X, \varphi)$ are called \emph{associative Smith maps}.  These provide $3$-harmonic, weakly conformal parametrizations of associative submanifolds.
\item Let $(X^7, g_X, \ast\varphi)$ be a $7$-manifold with a co-closed $\G_2$-structure $\ast\varphi \in \Omega^4(X)$.  The Smith immersions $f \colon (\Sigma^4, g_\Sigma, \vol_\Sigma) \to (X^7, g_X, \ast \varphi)$ are called \emph{coassociative Smith maps}.  These provide $4$-harmonic, weakly conformal parametrizations of coassociative submanifolds.
\end{enumerate}
Similar remarks and terminology applies to ``special Lagrangian Smith maps" into Calabi-Yau $2n$-manifolds, and to ``Cayley Smith maps" into $\Spin(7)$-manifolds, for example.
\end{example}

\subsection{The Poincar\'{e} metric on the $n$-ball} \label{sub:Hyperbolic-Space}

\indent \indent Finally, we recall the Poincar\'{e} metric on the $n$-ball.  Let $B^n(R) = \{x \in \R^n \colon |x| < R\}$ be the $n$-ball of radius $R > 0$.  On $B^n(R)$, the \emph{Poincar\'{e} metric} $g_R$ and \emph{Poincar\'{e} volume form} $\vol_R \in \Omega^n(B^n(R))$ are given by
\begin{align*}
g_R & = \frac{R^4}{(R^2 - |x|^2)^2}\, g_0 & \vol_R & = \left( \frac{R^2}{R^2 - |x|^2} \right)^n \vol_0,
\end{align*}
where $g_0$ and $\vol_0$ are the standard flat metric and volume form, respectively.  The scale factor in $g_R$ was chosen to ensure that for vectors $v \in T_0B^n(1)$ based at the origin, the Poincar\'{e} norm $|v|_1 = \sqrt{g_1(v,v)}$ equals the euclidean norm $|v|_0 = \sqrt{g_0(v,v)}$.  The following proposition summarizes some standard facts about the Poincar\'{e} metric that we will use in this work.

\begin{prop} \label{prop:Poincare-Facts} ${}$
\begin{enumerate}[(a)]
\item The Poincar\'{e} metric $g_R$ is conformal to the flat metric $g_0$.  Consequently:
\begin{itemize}
\item The identity map $\mathrm{Id} \colon (B^n(R), g_R, \vol_R) \to (B^n(R), g_0, \vol_0)$ is a Smith equivalence.
\item The inclusion map $\iota \colon (B^n(R), g_R, \vol_R) \hookrightarrow (\R^n, g_0, \vol_0)$ is a local Smith equivalence.
\end{itemize}
\item The Poincar\'{e} metric $g_R$ is complete and has constant sectional curvature $-\frac{4}{R^2}$.
\item The group of conformal automorphisms $\mathrm{ConfAut}(g_R)$ acts transitively and by isometries on $(B^n(R), g_R)$.
\item The induced distance function $\mathrm{dist}_R \colon B^n(R) \times B^n(R) \to \R$ is given by the explicit formula
$$\mathrm{dist}_R(x,y) = R \operatorname{arcsinh}  \sqrt{\frac{R^2 |x-y|^2}{ (R^2 - |x|^2)(R^2 - |y|^2) }}.$$
\end{enumerate}
\end{prop}

\section{$\phi$-Hyperbolicity}

\subsection{Definition and examples}

\indent \indent We now begin our study of hyperbolicity in earnest.  To get started, we generalize the notion of ``Brody hyperbolicity" from the K\"{a}hler setting to arbitrary calibrated manifolds.

\begin{defn}  Let $(X, g_X, \phi)$ be a calibrated $n$-manifold, where $\phi \in \Omega^k(X)$.  Say $(X, g_X)$ is \emph{$\phi$-hyperbolic} if every Smith immersion $(\R^k, g_0, \vol_0) \to (X, g_X, \phi)$ is constant.
\end{defn}
\noindent  When $g_X$ is clear from context, we will often simply say ``$X$ is $\phi$-hyperbolic."  The following is immediate.

\begin{prop} Let $(X, g_X, \phi)$ be a calibrated manifold.  If $X$ is $\phi$-hyperbolic, then $X$ contains no $\phi$-calibrated submanifold conformally diffeomorphic to $(\R^k, g_0)$.
\end{prop}

\noindent Note, however, that if $X$ is $\phi$-hyperbolic, there may exist $\phi$-calibrated submanifolds of $X$ that are conformally diffeomorphic to the flat $k$-ball $(B^k, g_0)$. Here are some basic properties of $\phi$-hyperbolicity.  The proofs are straightforward.

\pagebreak

\begin{prop} \label{prop:Basic} Let $(X, g_X, \phi)$ be a calibrated manifold.
\begin{enumerate}[(a)]
\item Suppose $(X, g_X, \phi)$ is Smith equivalent to $(Y, g_Y, \psi)$.  Then $(X, g_X)$ is $\phi$-hyperbolic if and only if $(Y, g_Y)$ is $\psi$-hyperbolic.
\item If $(X, g_X)$ is $\phi$-hyperbolic, then for every open set $U \subset X$, the manifold $(U, g_X|_U)$ is $\phi$-hyperbolic.
\end{enumerate}
\end{prop}

\indent We emphasize that the converse Proposition \ref{prop:Basic}(b) is false.  That is, if $(X, g_X)$ fails to be $\phi$-hyperbolic, then certain open subsets of $X$ could nevertheless be $\phi$-hyperbolic.  We now provide several examples and non-examples.

\begin{example}[K\"{a}hler calibrations]  Let $(X^{2m}, g_X, \omega)$ be a K\"{a}hler manifold.  Then:
$$(X, g_X) \text{ is }\omega\text{-hyperbolic} \ \iff \ X \text{ is Brody hyperbolic.}$$
Note that the Brody hyperbolicity of a K\"{a}hler manifold depends only on the complex structure (i.e., it is independent of the metric).
\begin{itemize}
\item In complex dimension one, a closed Riemann surface is Brody hyperbolic if and only if it has genus $\geq 2$.  A non-compact Riemann surface is Brody hyperbolic if and only if it is not biholomorphic to $\C$ or $\C^*$.
\item In a Stein manifold (e.g., $\C^n$), every bounded domain is Brody hyperbolic by Liouville's theorem.
\item Brody hyperbolicity is preserved by products, fibrations, and universal covers.
\item In $\CP^n$, a generic hypersurface of sufficiently large degree is Brody hyperbolic \cite{siu2015hyperbolicity}, \cite{brotbek2017hyperbolicity}.
\item In a complex torus, the Brody hyperbolic submanifolds are precisely those that do not contain a non-trivial sub-torus.
\item Hermitian manifolds with holomorphic sectional curvature bounded above by a negative constant are Brody hyperbolic.
\end{itemize}
\end{example}

\begin{example}[Flat $\R^n$ and $\R^n/\Z^n$] \label{example:Flat} Let $\phi \in \Lambda^k(\R^n)^*$ be a constant-coefficient calibration.  Let $T^n = \R^n/\Z^n$, and let $\pi \colon \R^n \to T^n$ be the projection.  By Proposition \ref{prop:Discrete}, there exists a calibration $\phi^\vee \in \Omega^k(T^n)$ having $\phi = \pi^*\phi^\vee$.
\begin{enumerate}[(a)]
\item The calibrated manifold $(\R^n, g_0, \phi)$ is \emph{not} $\phi$-hyperbolic.  In particular, $(\R^n, g_0, \vol_0)$ is \emph{not} $\vol_0$-hyperbolic.
\item The calibrated manifold $(T^n, g_0, \phi^\vee)$ is \emph{not} $\phi^\vee$-hyperbolic.
\end{enumerate}
\end{example}

\begin{example}[Volume form] Equip $(X^n, g_X)$ with the volume form calibration $\vol_X \in \Omega^n(X)$.
\begin{itemize}
\item If $\mathrm{Scal}_X \leq -A$ for some $A > 0$, then $(X, g_X)$ is $\mathrm{vol}_X$-hyperbolic.  We will prove this in Corollary \ref{cor:Vol-Hyperbolic}. Example \ref{example:Flat}(a) shows that this result cannot be improved to $A \geq 0$.
\item We remark that if $(X, g_X)$ is not $1$-quasiregularly elliptic, then $(X, g_X)$ is $\vol_X$-hyperbolic.  (We refer to \cite{bonk2001quasiregular} for a discussion of $K$-quasiregular ellipticity.)
\end{itemize}
\end{example}

\begin{example}[Quaternionic-K\"{a}hler calibrations] When $(X,g_X)$ is a quaternionic-K\"{a}hler $4m$-manifold, we use $\Psi \in \Omega^4(X)$ to denote the QK $4$-form.
\begin{itemize}
\item Quaternionic hyperbolic space $(\mathbb{HH}^n, g_{\mathbb{HH}^n})$ is $\Psi$-hyperbolic.  (See Corollary \ref{prop:HypSpacesRphi}(c).)
\item Quaternionic projective space $(\mathbb{HP}^n, g_{\mathbb{HP}^n})$ is \emph{not} $\Psi$-hyperbolic.  (Indeed, the standard embedding of $\mathbb{HP}^1$ minus a point can be parametrized by a Smith immersion from flat $\R^4$.)
\end{itemize}
\end{example}

\begin{example}[Associative calibrations] When $(X^7, g_X)$ is an oriented Riemannian $7$-manifold, we use $\varphi \in \Omega^3(X)$ to denote a $\G_2$-structure on $X$.
\begin{itemize}
\item As mentioned above, flat $\R^7$ and flat $T^7$ (equipped with the flat $\G_2$-structures $\varphi_0$) are \emph{not} $\varphi_0$-hyperbolic. However, every bounded open subset $U \subset \R^7$ is $\varphi_0$-hyperbolic (see Corollary \ref{cor:BoundedDomains}).
\item The Bryant-Salamon $7$-manifolds $\Lambda^2_-(S^4)$, $\Lambda^2_-(\CP^2)$, and $\Spinor(S^3)$ are \emph{not} $\varphi$-hyperbolic.  Indeed, all three contain associative submanifolds that are conformally diffeomorphic to flat $\R^3$.
\end{itemize}
Analogous remarks hold for Cayley calibrations.
\end{example}

\begin{example}[Trivial $\phi$-hyperbolicity] If a calibrated manifold $(X, g_X, \phi)$ admits no $\phi$-calibrated submanifolds, then $(X,g_X)$ is trivially $\phi$-hyperbolic.  (Note that this may occur even if $\mathrm{Gr}(\phi)$ is relatively large.)  Here are two examples of this.
\begin{itemize}
\item Let $(X^6, g_X)$ be a strict nearly-K\"{a}hler $6$-manifold, and let $(\omega, \Upsilon) \in \Omega^2(X) \oplus \Omega^3(X; \C)$ denote the $\SU(3)$-structure with the convention  that $d\omega = 3\,\mathrm{Re}(\Upsilon)$.  It is well-known that $X$ admits no $\mathrm{Re}(\Upsilon)$-calibrated $3$-folds (even locally), so $X$ is trivially $\mathrm{Re}(\Upsilon)$-hyperbolic.
\item Let $(X^7, g_X)$ be a $7$-manifold with a nearly-parallel $\G_2$-structure, and let $\psi \in \Omega^4(X)$ be the coassociative $4$-form.  It is well-known that $X$ admits no coassociative $4$-folds (even locally), so $X$ is trivially $\psi$-hyperbolic.
\end{itemize}
\end{example}

\subsection{Domains in $\R^n$}

\indent \indent For domains $U \subset \R^n$ equipped with the flat metric, we now make some basic observations about $\phi$-hyperbolicity.  First, we set terminology.

\begin{defn} A \emph{domain in $\R^n$} is a non-empty, open, connected subset $U \subset \R^n$.  When a calibration $\phi \in \Omega^k(\R^n)$ is specified, we will typically regard a domain $U \subset \R^n$ as the calibrated manifold $(U, g_0, \phi)$.  (An important exception to this is the open $n$-ball $B^n$, which we typically equip with the Poincar\'{e} metric $g_1$.)  Two domains in $\R^n$ are \emph{congruent} if they differ by a rigid motion of $\R^n$ (i.e., an element of the euclidean group $\mathrm{E}(n) = \R^n \rtimes \mathrm{O}(n)$.)
\end{defn}

\begin{rmk} The property of ``$\phi$-hyperbolicity" is generally not preserved by ambient isometries.  In particular, if $U_1, U_2 \subset \R^n$ are congruent domains, the $\phi$-hyperbolicity of $U_1$ need not imply the $\phi$-hyperbolicity of $U_2$.   For illustration, consider $(\R^4, g_0)$ with coordinates $(x_1, y_1, x_2, y_2)$, and let $\omega = dx_1 \wedge dy_1 + dx_2 \wedge dy_2 \in \Omega^2(\R^4)$ be the standard K\"{a}hler form.  Then the domains
\begin{align*}
U_1 & = \{ (x_1)^2 + (x_2)^2 < 1 \}, & U_2 & = \{ (x_1)^2 + (y_1)^2 < 1 \}
\end{align*}
are congruent.  Since $U_2$ contains a complex line, we see that $U_2$ is not $\omega$-hyperbolic (i.e., Brody hyperbolic).  On the other hand, since $U_1$ is convex and does not contain any complex lines, a theorem of Barth \cite{MR572300} implies that $U_1$ is $\omega$-hyperbolic (i.e., Brody hyperbolic).
\end{rmk}

\indent The following definitions were first introduced in \cite{ikonen2024liouville}.

\begin{defn} Let $\phi \in \Lambda^k(\R^n)^*$ be a constant-coefficient calibration.
\begin{itemize}
\item Say $\phi$ is \emph{conformally rigid} if every conformally flat, $\phi$-calibrated submanifold of $\R^n$ is flat.
\item Let $U \subset \R^k$ be a domain.  An \emph{inner M\"{o}bius curve} is a map $f \colon U \to \R^n$ of the form $f(x) = y_0 + L(v(x))$, where $y_0 \in \R^n$, $L \colon \R^k \to \R^n$ is an orthogonal linear map, and $v \colon U \to \R^k$ is the restriction of a non-constant M\"{o}bius transformation $S^k \to S^k$, where we view $S^k = \R^k \cup \{\infty\}$.
\item Say $\phi$ is \emph{inner M\"{o}bius rigid} if every non-constant Smith immersion $f \colon (U, g_0, \vol_0) \to (\R^n, g_0, \phi)$ is an inner M\"{o}bius curve.
\end{itemize}
\end{defn}

\begin{thm}[\cite{ikonen2024liouville}] Let $\phi \in \Lambda^k(\R^n)^*$ be a constant-coefficient calibration, $k \geq 3$.  Then $\phi$ is conformally rigid if and only if $\phi$ is inner M\"{o}bius rigid.
\end{thm}

\begin{example} \label{ex:ConfRigid} The following remarks are drawn from \cite{ikonen2024liouville}.
\begin{enumerate}[(a)]
\item For a given $(k,n)$, the collection of inner M\"{o}bius rigid calibrations is an open dense subset of the set of all constant-coefficient calibrations in $\Lambda^k(\R^n)^*$.
\item The volume form $\vol_0 \in \Lambda^n(\R^n)^*$ is inner M\"{o}bius rigid, essentially by Liouville's theorem on conformal maps.
\item For $p \geq 2$, the K\"{a}hler calibrations $\frac{1}{p!}\omega^p \in \Omega^{2p}(\R^{2n})$ are inner M\"{o}bius rigid.
\item By contrast, the K\"{a}hler calibration $\omega \in \Lambda^2(\R^{2n})^*$, the special Lagrangian calibration $\mathrm{Re}(dz_1 \wedge dz_2 \wedge dz_3) \in \Lambda^3(\R^6)^*$, the associative calibration $\varphi \in \Lambda^3(\R^7)^*$, and the Cayley calibration $\Phi \in \Lambda^4(\R^8)^*$ are \emph{not} inner M\"{o}bius rigid.  (The authors do not know whether the coassociative calibration $\ast\varphi \in \Lambda^4(\R^7)^*$ is inner M\"{o}bius rigid.)
\end{enumerate}
\end{example}

\indent When $\phi \in \Lambda^k(\R^n)^*$ is inner M\"{o}bius rigid, we can completely characterize those domains $U \subset \R^n$ that are $\phi$-hyperbolic.

\begin{thm} \label{prop:ConfRigid} Let $\phi \in \Lambda^k(\R^n)^*$ be a constant-coefficient calibration, $k \geq 2$.  Let $U \subset \R^n$ be a domain, and regard $U$ as the calibrated manifold $(U, g_0, \phi)$.
\begin{enumerate}[(a)]
\item If $U$ is $\phi$-hyperbolic, then $U$ contains no affine $\phi$-planes.
\item Suppose $\phi$ is inner M\"{o}bius rigid.  Then the converse of (a) holds.
\end{enumerate}
\end{thm}

\noindent \textit{Proof:}
\begin{enumerate}[(a)]
\item Suppose there exists an affine $\phi$-plane $E \subset U$.  Then there exists a non-constant Smith immersion $(\R^k, g_0, \vol_0) \to (U, g_0, \phi)$ whose image is $E$, so $U$ is not $\phi$-hyperbolic.
\item Suppose $\phi$ is inner M\"{o}bius rigid.  Suppose that $U$ is not $\phi$-hyperbolic, so there exists a non-constant Smith immersion $F \colon (\R^k, g_0, \vol_0) \to (U, g_0, \phi)$.  Let $\iota \colon U \hookrightarrow \R^n$ be the inclusion, so that $f := \iota \circ F \colon (\R^k, g_0, \vol_0) \to (\R^n, g_0, \phi)$ is a Smith immersion.  Since $f$ is non-constant, inner M\"{o}bius rigidity implies that $f(x) = y_0 + L(v(x))$ for some $y_0 \in \R^n$, some orthogonal linear map $L \colon \R^k \to \R^n$, and some map $v \colon \R^k \to \R^k$ that is the restriction of a M\"{o}bius transformation $S^k \to S^k$.  Since $v$ is entire (defined on all of $\R^k$) and the restriction of a M\"{o}bius transformation, $v$ is surjective.  Therefore, $\mathrm{Image}(f) = y_0 + \mathrm{Image}(L \circ v)$ is an affine $k$-plane in $U$, and it is $\phi$-calibrated because $F$ is Smith.  That is, $U$ contains an affine $\phi$-plane.
\end{enumerate}

\pagebreak

\begin{cor}  ${}$
\begin{enumerate}[(a)]
\item Let $U \subset \R^{2n}$ be a domain, and let $\phi = \frac{1}{p!}\omega^p \in \Omega^{2p}(\R^{2n})$ have $p \geq 2$.  Then $(U,g_0)$ is $\frac{1}{p!}\omega^p$-hyperbolic if and only if $U$ contains no complex $p$-planes.
\item Let $U \subset \R^n$ be a domain, and let $\phi = \vol_0 \in \Omega^n(U)$.  Then $(U, g_0)$ is $\vol_0$-hyperbolic if and only if $U \neq \R^n$.
\end{enumerate}
\end{cor}
\begin{proof} This follows from Theorem \ref{prop:ConfRigid}(b) and Example \ref{ex:ConfRigid}.
\end{proof}

\section{Schwarz lemmas and consequences}

\indent \indent Our primary goal in this section is to derive a Schwarz lemma for weakly conformal $k$-harmonic maps in general, and for Smith immersions in particular.  In $\S$\ref{sub:PhiSec}, we define a notion of ``$\phi$-sectional curvature" that extends the holomorphic sectional curvature to arbitrary calibrated geometries $(X, g_X, \phi)$.  Then, in $\S$\ref{sub:Schwarz} we derive a Bochner formula for weakly conformal $k$-harmonic maps and Smith immersions.  This formula, together with a maximum principle argument, will imply the Schwarz lemma (Theorem \ref{thm:SchwarzMain}).  Finally, in $\S$\ref{sub:Consequences}, we draw several consequences.  In particular, we prove (Theorem \ref{cor:NegativeCurvature}) that strongly negative $\phi$-sectional curvature implies $\phi$-hyperbolicity.

\subsection{$\phi$-Sectional curvature} \label{sub:PhiSec}

\begin{defn} Let $(X, g_X, \phi)$ be a semi-calibrated $n$-manifold, where $\phi \in \Omega^k(X)$ with $k \geq 2$.  The \emph{$\phi$-sectional curvature} of a $\phi$-plane $\xi \in \mathrm{Gr}(\phi)$ is
$$\mathrm{Sec}^\phi(\xi) := \sum_{1 \leq i,j \leq k} \mathrm{Sec}(e_i \wedge e_j) = \sum_{1 \leq i,j \leq k} R_{ijji}$$
where $\{e_1, \ldots, e_k\}$ is an orthonormal basis of $\xi$.  One can check that this definition is well-defined (independent of basis), and that it is independent of the orientation of $\xi$.
\end{defn}

\begin{example} \label{example:Sectional} ${}$
\begin{itemize}
\item Let $\phi = \vol_X \in \Omega^n(X)$ be the volume form.  Then $\mathrm{Sec}^{\vol_X}$ is the scalar curvature of $X$.
\item Let $\phi = \omega \in \Omega^2(X)$ be a K\"{a}hler form on a K\"{a}hler manifold $X$.  Then $\mathrm{Sec}^{\omega}(\xi)$ is the holomorphic sectional curvature of a complex line $\xi \subset TX$.
\item Let $\phi = \Psi \in \Omega^4(X)$ be a QK $4$-form on a quaternionic-K\"{a}hler manifold $X$.  Then $\mathrm{Sec}^\Psi(\xi)$ extends the notion of ``$Q$-sectional curvature" \cite{ishihara1973quaternion}. \\

To explain this, for $u \in T_xX$, we let $Q(u) \subset T_xX$ be the quaternionic line spanned by $u$.  If $\mathrm{Sec}(v_1 \wedge v_2)$ is constant on non-zero elements $v_1 \wedge v_2 \in \Lambda^2(Q(u)) \subset \Lambda^2(T_xX)$, then $\rho(u) := \mathrm{Sec}(v_1 \wedge v_2)$ is called the $Q$-sectional curvature of $u$, and one can calculate that $\mathrm{Sec}^\Psi(Q(u)) = 12\,\rho(u)$.  However, if $\mathrm{Sec}(v_1 \wedge v_2)$ depends on the choice of non-zero $v_1 \wedge v_2 \in \Lambda^2(Q(u))$, then the $Q$-sectional curvature of $u$ is undefined (whereas $\mathrm{Sec}^\Psi(Q(u))$ still makes sense). \\

Note that $\mathrm{Sec}^\Psi$ is not the notion of ``quaternionic sectional curvature" considered by Moroianu-Semmelmann-Weingart \cite{moroianu2024quaternion}, nor is it the ``quaternionic bisectional curvature" considered by Macia-Semmelmann-Weingart \cite{macia2025quaternionic}.
\end{itemize}
\end{example}

\indent We now make two remarks about this definition.  First, when the ambient metric $g_X$ is Einstein (as happens in special holonomy settings), there is a simple relationship between the $\phi$- and $(\ast \phi)$-sectional curvature.

\begin{prop} Let $(X^n, g_X, \phi)$ be a semi-calibrated $n$-manifold, $\phi \in \Omega^k(X)$.  If $\mathrm{Ric}(g_X) = c g_X$ for some constant $c \in \R$, then
$$\mathrm{Sec}^\phi(\xi) = (2k-n)c + \mathrm{Sec}^{\ast \phi}(\ast \xi), \ \ \forall \xi \in \mathrm{Gr}(\phi).$$
In particular, if $g_X$ is Ricci-flat, then $\mathrm{Sec}^\phi(\xi) = \mathrm{Sec}^{\ast \phi}(\ast \xi)$ holds for all $\phi$-planes $\xi$.
\end{prop}

\begin{proof} Let $\{e_1, \ldots, e_n\}$ be an orthonormal frame such that $\xi = \mathrm{span}\{e_1, \ldots, e_k\}$ and also $\ast \xi = \mathrm{span}\{e_{k+1}, \ldots, e_n\}$.  We adopt index ranges $1 \leq i,j \leq k$ and $k+1 \leq \overline{i}, \overline{j} \leq n$ and $1 \leq A,B,C \leq n$.  In particular,
\begin{align*}
\mathrm{Sec}^\phi(\xi) & = \sum_{i,j} R_{ijji}, & \mathrm{Sec}^{\ast \phi}(\ast \xi) & = \sum_{\bar{i},\bar{j}} R_{\bar{i} \bar{j} \bar{j} \bar{i}}, & \displaystyle \operatorname{Ric}_{AB} & = \sum_C R_{CABC}\ (=c\delta_{AB}).
\end{align*}
The last of these implies $\displaystyle \sum_i R_{iABi} = c \delta_{AB}- \sum_{\bar{i}} R_{\bar{i} AB \bar{i}}$. Hence,
\begin{align*}
    \mathrm{Sec}^\phi(\xi) & = \sum_{i,j} R_{ijji} = \sum_j c\delta_{jj}- \sum_{\bar{i},j} R_{\bar{i} j j\bar{i}} = kc - \sum_{\bar{i},j} R_{j \bar{i}\bar{i} j} = kc -\Big(c\sum_{\bar{i}} \delta_{\bar{i} \bar{i}} - \sum_{\bar{i},\bar{j}} R_{\bar{j} \bar{i} \bar{i} \bar{j}} \Big)\\
    &= kc - (n-k)c + \mathrm{Sec}^{\ast \phi}(\ast \xi) = (2k-n)c + \mathrm{Sec}^{\ast \phi}(\ast \xi).
\end{align*}
\end{proof}

Next, we observe that in certain cases of interest, a sign on the $\phi$-sectional curvature implies a sign on the scalar curvature.  For example:

\begin{prop} ${}$
\begin{enumerate}[(a)]
\item Let $(X^7, g_X, \varphi)$ be a $7$-manifold with a $\G_2$-structure.
\begin{itemize}
\item If $\mathrm{Sec}^\varphi_X > 0$ (resp., $\geq 0, < 0, \leq 0$), then $\mathrm{Scal}_X > 0$ (resp., $\geq 0, < 0, \leq 0$).
\item If $\mathrm{Sec}^{\ast \varphi}_X > 0$ (resp., $\geq 0, < 0, \leq 0$), then $\mathrm{Scal}_X > 0$ (resp., $\geq 0, < 0, \leq 0$).
\end{itemize}
\item Let $(X^8, g_X, \Phi)$ be an $8$-manifold with a $\Spin(7)$-structure.  If $\mathrm{Sec}^\Phi_X > 0$ (resp., $\geq 0, < 0, \leq 0$), then $\mathrm{Scal}_X > 0$ (resp., $\geq 0, < 0, \leq 0$).
\end{enumerate}
\end{prop}

\begin{proof} We prove part (a).  Fix $p \in X$, and let $(e_1, \ldots, e_7)$ be $\mathrm{G}_2$-adapted frame at $p$, so that $\varphi|_p = e^{123} + e^{145} + e^{167} + e^{246} - e^{257} - e^{347} - e^{356}$.  Then at $p \in X$, we have
$$\mathrm{Sec}^\varphi(e^{123}) + \mathrm{Sec}^\varphi(e^{145}) + \cdots + \mathrm{Sec}^\varphi(-e^{356}) = \sum_{i,j} R_{ijji} = \mathrm{Scal}_X(p),$$
which gives the result. In the case of $\ast \varphi$, the same procedure yields $2\,\mathrm{Scal}_X(p)$ on the right-hand side.  In the $\Spin(7)$ case, one obtains $3\,\mathrm{Scal}_X(p)$ on the right.
\end{proof}

\subsection{A Schwarz lemma for Smith immersions} \label{sub:Schwarz}

\indent \indent We now aim to establish a suitable Bochner formula.  For this, let $f \colon (\Sigma^k, g_\Sigma) \to (X^n, g_X)$ be an arbitrary smooth map between Riemannian manifolds.  The following formula is well-known:
\begin{equation}
\frac{1}{2}\Delta |df|^2 = \left| \nabla df \right|^2 + \langle \nabla \tau_2(f), df \rangle + R^{\Sigma, X}(f), \label{eq:ClassicBochner}
\end{equation}
where here
\begin{align*}
R^{\Sigma, X}(f) & = \sum_i \langle df(\mathrm{Ric}_\Sigma e_i), df(e_i) \rangle - \sum_{i,j} R^X( df(e_i), df(e_j), df(e_j), df(e_i)),
\end{align*}
and $\{e_1, \ldots, e_k\}$ is a local orthonormal frame of $T\Sigma$.  When $f$ is weakly conformal, we can simplify the curvature term in the following way:

\begin{lem} \label{prop:Classic-Bochner} If $f \colon (\Sigma^k, g_\Sigma) \to (X^n, g_X)$ is a weakly conformal map, then
\begin{equation} \label{eq:ConformalCurvature}
R^{\Sigma, X}(f) = \frac{|df|^2}{k}\mathrm{Scal}_\Sigma - \frac{|df|^4}{k^2} \sum_{i,j} \mathrm{Sec}_X( v_i \wedge v_j),
\end{equation}
where $\{v_i\}$ is a local orthonormal frame for $df(T\Sigma) \subset f^*TX$.
\end{lem}

\begin{proof} Since $f$ is weakly conformal, the first term of $R^{\Sigma, X}(f)$ is
\begin{align*}
 \sum_i \langle df(\mathrm{Ric}_\Sigma e_i), df(e_i) \rangle = \sum_{i} (f^*g_X)(\mathrm{Ric}_\Sigma e_i, e_i) & = \frac{|df|^2}{k} \sum_i \mathrm{Ric}_\Sigma(e_i, e_i) = \frac{|df|^2}{k}\,\mathrm{Scal}_\Sigma.
 \end{align*}
 For the second term, we may take $v_i = \frac{1}{\lambda} df(e_i)$ as the orthonormal frame for $df(T\Sigma) \subset f^*TX$, where $\lambda = \frac{1}{\sqrt{k}}|df|$ is the conformal factor.  Then
 \begin{align*}
\sum_{i,j} R^X( df(e_i), df(e_j), df(e_j), df(e_i)) & = \sum_{i,j} \left| df(e_i) \wedge df(e_j) \right|^2 \mathrm{Sec}_X(df(e_i) \wedge df(e_j)) \\
& = \lambda^4 \sum_{i,j} \mathrm{Sec}_X(v_i \wedge v_j) \\
& = \frac{|df|^4}{k^2}\sum_{i,j} \mathrm{Sec}_X(v_i \wedge v_j).
\end{align*}
\end{proof}

Specializing further to weakly conformal $k$-harmonic maps, we will prove:

\begin{prop}[Bochner formula] \label{prop:Bochner} Let $f \colon (\Sigma^k, g_\Sigma, \vol_\Sigma) \to (X^n, g_X, \phi)$ be a weakly conformal, $k$-harmonic map, where $\phi \in \Omega^k(X)$ is a calibration.  Then
\begin{equation}
\frac{k-1}{k^2} \Delta |df|^{2k} = |df|^{2k-2} |\nabla df|^2 + \frac{3k-4}{4} \left| df \right|^{2k-4} \left| \nabla |df|^2 \right|^2 + \frac{|df|^{2k}}{k}\,\mathrm{Scal}_\Sigma - \frac{|df|^{2k+2}}{k^2} \sum_{i,j} \mathrm{Sec}_X(v_i \wedge v_j),
\end{equation}
where $\{v_i\}$ is a local orthonormal frame for $df(T\Sigma) \subset f^*TX$.  Moreover, if $f$ is a Smith immersion, then the last sum is $\sum_{i,j} \mathrm{Sec}_X(v_i \wedge v_j) = \mathrm{Sec}^\phi_X( df(T\Sigma))$.
\end{prop}

Before beginning the proof, we recall the following standard formula: for any smooth map $h \colon (M^m, g_M) \to (N^n, g_N)$ between Riemannian manifolds, and for any $k \geq 2$, we have
\begin{equation} \label{kTension2Tension}
\tau_k(h) = |dh|^{k-2} \tau_2(h) + \frac{k-2}{2} |dh|^{k-4} \langle dh, \nabla |dh|^2 \rangle.
\end{equation}

\begin{proof} Let $f \colon (\Sigma^k, g_\Sigma, \vol_\Sigma) \to (X^n, g_X, \phi)$ be a weakly conformal, $k$-harmonic map.  Applying the chain rule to $(|df|^2)^k$ and to $(|df|^k)^2$, we have
\begin{align}
\frac{1}{2k} \Delta |df|^{2k} & = \frac{1}{2} |df|^{2k-2} \Delta |df|^2 + \frac{k-1}{2} |df|^{2k-4} \left| \nabla |df|^2 \right|^2 \label{eq:Chain1} \\
\frac{1}{2} \Delta |df|^{2k} & = |df|^k \Delta |df|^k + \frac{k^2}{4} |df|^{2k-4} \left| \nabla |df|^2 \right|^2. \label{eq:Chain2}
\end{align}
Substituting the Bochner formula (\ref{eq:ClassicBochner}) into equation (\ref{eq:Chain1}), we obtain
\begin{equation}
\frac{1}{2k} \Delta |df|^{2k} = \langle \nabla \tau_2(f), |df|^{2k-2}df \rangle + |df|^{2k-2} | \nabla df |^2 + \frac{k-1}{2}|df|^{2k-4} \left| \nabla |df|^2 \right|^2 + |df|^{2k-2} R^{\Sigma, X}(f).
\label{eq:BochnerPrelim}
\end{equation}
Now, if $p \in \Sigma$ is a critical point of $f$ (so that $df|_p = 0$), equation (\ref{eq:BochnerPrelim}) implies that $(\Delta |df|^{2k})(p) = 0$.  Therefore, the equation we are trying to prove is trivially true at critical points.  So, from now on, we work on the open set $\mathrm{Reg}(f) = \{p \in \Sigma \colon \mathrm{rank}(df_p) = k\}$.  Note that $|df|$ is smooth on $\mathrm{Reg}(f)$. \\
\indent We will calculate the first term in (\ref{eq:BochnerPrelim}).  For this, let $(e_i)$ and $(v_\alpha)$ be local orthonormal frames on $\Sigma$ and $X$, respectively, and write $df(e_i) = f^\alpha_i v_\alpha$.  Since $f$ is weakly conformal, we have
\begin{equation} \label{eq:ConformalCoordinates}
f^\alpha_i f^\alpha_j = \lambda^2 \delta_{ij} = \frac{1}{k}|df|^2 \delta_{ij}.
\end{equation}
Also, since $f$ is $k$-harmonic, equation (\ref{kTension2Tension}) gives
\begin{equation}
|df|^{k-2} \tau_2(f) = -\frac{k-2}{2} |df|^{k-4} \langle df, \nabla |df|^2 \rangle.
\label{eq:2tension}
\end{equation}
By definition, $k$-harmonic maps satisfy $\mathrm{div}( |df|^{k-2} df) = 0$.  Therefore, applying the Leibniz rule to this equation, we calculate
\begin{align*}
\langle \nabla \tau_2(f), |df|^{k-2} df \rangle & = \mathrm{div}(   |df|^{k-2}  \langle\tau_2(f), df \rangle) \\
& = -\frac{k-2}{2}\,\mathrm{div}\!\left( |df|^{k-4} \left\langle \langle df, \nabla |df|^2 \rangle, df \right\rangle \right) \tag{\text{by }(\ref{eq:2tension})} \\
& = -\frac{k-2}{2} \nabla_j \!\left( |df|^{k-4} f^\alpha_i f^\alpha_j \nabla_i |df|^2  \right)  \\
& = -\frac{k-2}{2k} \nabla_i \!\left( |df|^{k-2} \nabla_i |df|^2 \right) \tag{\text{by } (\ref{eq:ConformalCoordinates})}  \\
& = -\frac{k-2}{k^2} \nabla_i \nabla_i (|df|^2)^{k/2}  \\
& = -\frac{k-2}{k^2} \Delta |df|^k.
\end{align*}
Therefore, the first term in (\ref{eq:BochnerPrelim}) is
\begin{align}
\langle \nabla \tau_2(f), |df|^{2k-2}df \rangle & = -\frac{k-2}{k^2} |df|^k \Delta |df|^k \notag \\
& = -\frac{k-2}{k^2} \left( \frac{1}{2} \Delta |df|^{2k} - \frac{k^2}{4} |df|^{2k-4} \left| \nabla |df|^2 \right|^2  \right) \tag{\text{by }(\ref{eq:Chain2})} \\
& = -\frac{k-2}{2k^2} \Delta |df|^{2k} +  \frac{k-2}{4}  |df|^{2k-4} \left| \nabla |df|^2 \right|^2. \label{eq:FirstTerm}
\end{align}
Finally, substituting (\ref{eq:FirstTerm}) into (\ref{eq:BochnerPrelim}) and rearranging gives
$$\frac{k-1}{k^2} \Delta |df|^{2k} = |df|^{2k-2} | \nabla df |^2 + \frac{3k-4}{4}|df|^{2k-4} \left| \nabla |df|^2 \right|^2 + |df|^{2k-2} R^{\Sigma, X}(f).$$
Applying (\ref{eq:ConformalCurvature}) gives the result.
\end{proof}

\indent Now, when $\Sigma$ is compact, the above Bochner formula quickly implies the Schwarz lemma.  Indeed, the maximum principle implies that a point where $|df|^{2k}$ attains its maximum value, we have $\Delta |df|^{2k} \leq 0$.  Consequently, we obtain
$$ 0 \geq \frac{k-1}{k^2}  \Delta|df|^{2k} \geq |df|^{2k-2} R^{\Sigma, X}(f) = |df|^{2k}\Big( \frac{1}{k}\mathrm{Scal}_\Sigma - \frac{|df|^2}{k^2} \sum_{i,j} \mathrm{Sec}_X( v_i \wedge v_j) \Big).$$
Imposing appropriate curvature bounds then yields an upper bound on $|df|^2$. \\
\indent On the other hand, when $\Sigma$ is non-compact, the function $|df|^2$ need not be bounded \emph{a priori}, or the supremum need not be achieved.  However, by an application of Yau's maximum principle, we are able to extend this argument to complete manifolds $(\Sigma, g_\Sigma)$ satisfying a lower Ricci curvature bound.  We now explain this.

\begin{lem}[Yau's maximum principle \cite{yau1975harmonic}] \label{lem:YauMax} Let $(\Sigma, g_\Sigma)$ be a complete Riemannian manifold with $\mathrm{Ric}_\Sigma \geq -E$ for some constant $E > 0$.  If $v \in C^\infty(\Sigma)$ is bounded below, then for any $\epsilon > 0$, there exists a point $p_\epsilon \in \Sigma$ such that $|(\nabla v)(p_\epsilon)| < \epsilon$, $(\Delta v)(p_\epsilon) > -\epsilon$, and $v(p_\epsilon) < \inf(v) + \epsilon$.
\end{lem}

\indent The following result appears to be well-known to experts (see, e.g., \cite[Corollary 3.9]{albanese2017schwarz}), but we were unable to locate a proof in the literature.  For completeness, we supply a proof.

\begin{prop} \label{prop:MaximumArgument} Let $(\Sigma, g_\Sigma)$ be a complete Riemannian manifold with $\mathrm{Ric}_\Sigma \geq -E$ for some constant $E > 0$.  Let $u \in C^\infty(\Sigma)$ be a non-negative function. If there exist constants $A \geq 0$, $B > 0$, $c > 0$ such that
$\Delta u \geq -Au + Bu^{1+c}$, then $u$ is bounded and $u^c \leq \frac{A}{B}$. 
\end{prop}

\begin{proof} Let $u \in C^\infty(\Sigma)$ be a non-negative function for which $\Delta u \geq -Au + Bu^{1+c}$, where $A \geq 0$ and $B,c > 0$.  Let $v = (1+u)^{-a}$, where $0 < a < \frac{1}{2}c$.  Then $v \in C^\infty(\Sigma)$ and $v \geq 0$.  Using the shorthand $v' := -a(1+u)^{-a-1}$ and $v'' := a(1+a)(1+u)^{-a-2}$, we have $\nabla v = v' \nabla u$ and $\Delta v = v' \Delta u + v'' |\nabla u|^2$.  Therefore,
\begin{align*}
\Delta u = \frac{1}{v'} \Delta v - \frac{v''}{v'} |\nabla u|^2 & = \frac{1}{v'} \Delta v - \frac{v''}{(v')^3} |\nabla v|^2 \\
& = -\frac{1}{a} (1+u)^{a+1} \Delta v + \frac{1+a}{a^2}(1+u)^{1+2a} |\nabla v|^2.
\end{align*}
Now, let $\epsilon > 0$.  By Yau's maximum principle (Lemma \ref{lem:YauMax}), there exists $p_\epsilon \in \Sigma$ such that $|(\nabla v)(p_\epsilon)| < \epsilon$, $(\Delta v)(p_\epsilon) > -\epsilon$, and $v(p_\epsilon) < \inf(v) + \epsilon$.  Therefore, at the point $p_\epsilon$, we have
\begin{align} \label{eq:PrelimBound}
-Au + Bu^{1+c} \leq \Delta u & <  \frac{1}{a} (1+u)^{1+a} \epsilon + \frac{1+a}{a^2}(1+u)^{1+2a} \epsilon^2 \notag \\
& < \epsilon(1+u)^{1+c} \left( \frac{1}{a} + \frac{1+a}{a^2} \epsilon \right)\!.
\end{align}
\indent Now, suppose for contradiction that $u$ were unbounded.  Then $\inf(v) = 0$, so $v(p_\epsilon) < \epsilon$, so $u(p_\epsilon) > (\frac{1}{\epsilon})^{1/a} - 1 \to \infty$ as $\epsilon \to 0$.  Now, the bound (\ref{eq:PrelimBound}) gives
\begin{equation}
-A \frac{u}{(1+u)^{1+c}} + B\left( \frac{u}{1+u} \right)^{1+c} < \epsilon \left( \frac{1}{a} + \frac{1+a}{a^2} \epsilon \right)\!,
\end{equation}
so sending $\epsilon \to 0$ gives $B \leq 0$, contrary to assumption.  Thus, $u$ is bounded, so $\sup(u)$ is finite.  Finally, letting $\epsilon \to 0$ in (\ref{eq:PrelimBound}) and noting that $u(p_\epsilon) \to \sup(u)$, we obtain $-A \sup(u) + B \sup(u)^{1+c} \leq 0$.  Rearranging gives $\sup(u)^c \leq \frac{A}{B}$, which yields the claim.
\end{proof}

\indent We arrive at the main result of this section:

\begin{thm}[Schwarz Lemma] \label{thm:SchwarzMain} Let $f \colon (\Sigma^k, g_\Sigma, \vol_\Sigma) \to (X^n, g_X, \phi)$ be a smooth map, $\phi \in \Omega^k(X)$.  Suppose that $\Sigma$ is complete, and that $\mathrm{Ric}_\Sigma \geq -E$ for some $E > 0$.
\begin{enumerate}[(a)]
\item Suppose $f$ is weakly conformal and $k$-harmonic.  If $\mathrm{Scal}_\Sigma \geq -S$ for $S \geq 0$, and if $\mathrm{Sec}_X \leq -B$ for $B > 0$, then $|df|^2 \leq \frac{S}{(k-1)B}$.
\item Suppose $f$ is a Smith immersion.  If $\mathrm{Scal}_\Sigma \geq -S$ for $S \geq 0$, and if $\mathrm{Sec}^\phi_X \leq -A$ for $A > 0$, then $|df|^2 \leq \frac{kS}{A}$.
\end{enumerate}
\end{thm}

\begin{proof} Suppose $f \colon \Sigma \to X$ is weakly conformal and $k$-harmonic.  The Bochner formula of Proposition \ref{prop:Bochner} gives
\begin{align} \label{eq:Bochner-Step}
\frac{k-1}{k^2} \Delta |df|^{2k} & \geq \frac{|df|^{2k}}{k}\,\mathrm{Scal}_\Sigma - \frac{|df|^{2k+2}}{k^2} \sum_{i,j} \mathrm{Sec}_X(v_i \wedge v_j).
\end{align}
\begin{enumerate}[(a)]
\item If $\mathrm{Scal}_\Sigma \geq -S$ and $\mathrm{Sec}_X \leq -B$, then the bound (\ref{eq:Bochner-Step}) gives
\begin{align*}
\frac{k-1}{k^2} \Delta |df|^{2k} & \geq -\frac{S}{k} |df|^{2k} + \frac{B(k-1)}{k} \left( |df|^{2k} \right)^{1 + \frac{1}{k}}.
\end{align*}
Applying Proposition \ref{prop:MaximumArgument} to the function $u = |df|^{2k}$ gives the result.
\item If $\mathrm{Scal}_\Sigma \geq -S$ and $\mathrm{Sec}^\phi_X \leq -A$, recalling that $\sum \mathrm{Sec}_X(v_i \wedge v_j) = \mathrm{Sec}^\phi_X( df(T\Sigma))$, we see that (\ref{eq:Bochner-Step}) gives
\begin{align*}
\frac{k-1}{k^2} \Delta |df|^{2k} & \geq -\frac{S}{k}|df|^{2k} + \frac{A}{k^2} (|df|^{2k})^{1 + \frac{1}{k}}.
\end{align*}
Applying Proposition \ref{prop:MaximumArgument} to the function $u = |df|^{2k}$ gives the result.
\end{enumerate}
\end{proof}

\subsection{Consequences} \label{sub:Consequences}

\indent \indent We now draw several consequences from the Schwarz lemma.  To do so, we need the following fact about weakly conformal maps.

\begin{lem}[Basic distance estimate] \label{lem:BasicDistance} Let $f \colon (\Sigma^k, g_\Sigma) \to (X^n, g_X)$ be a weakly conformal map such that $|df|^2 \leq D$.  Then for all $p \in \Sigma$ and $w \in T_p\Sigma$ with $|w| = 1$, we have $\left| df_p(w) \right| \leq \sqrt{ \frac{D}{k} }$.  Moreover,
$$\mathrm{dist}_X(f(p), f(q)) \leq \sqrt{ \frac{D}{k} }\,\mathrm{dist}_\Sigma(p,q), \ \ \forall p,q \in \Sigma.$$
\end{lem}

\begin{proof} Let $f \colon \Sigma \to X$ be weakly conformal, fix $p \in \Sigma$, and let $\lambda(p)$ denote the conformal factor of $f$ at $p$.  Then for all $w \in T_p\Sigma$ with $|w| = 1$, we have (using (\ref{eq:Lambda})) $|df_p(w)| = \lambda(p) |w| = \lambda(p) = \frac{1}{\sqrt{k}}|df_p| \leq \sqrt{\frac{D}{k}}$.  Now, fix $p,q \in \Sigma$, and consider the following two sets of piecewise-smooth curve segments:
\begin{align*}
C_\Sigma & = \left\{ \widetilde{\gamma} \colon [0,1] \to \Sigma \colon \widetilde{\gamma}(0) = p, \widetilde{\gamma}(1) = q \right\} \\
C_X & = \left\{ \gamma \colon [0,1] \to X \colon \gamma(0) = f(p), \gamma(1) = f(q) \right\}\!.
\end{align*}
Notice that $C_X$ contains the class of curve segments $\{f \circ \widetilde{\gamma} \colon [0,1] \to X \mid \widetilde{\gamma} \in C_\Sigma\}$.  Therefore,
\begin{align*}
\mathrm{dist}_X(f(p), f(q)) = \inf_{\gamma \in C_X} \int_0^1 | \gamma'(t) |_X\,dt & \leq \inf_{\widetilde{\gamma} \in C_\Sigma} \int_0^1 \left| (f \circ \widetilde{\gamma})'(t) \right|_X\,dt \\
& = \inf_{\widetilde{\gamma} \in C_\Sigma} \int_0^1 \lambda(\widetilde{\gamma}(t)) \left| \widetilde{\gamma}'(t) \right|_\Sigma dt \\
& \leq \sqrt{\frac{D}{k}} \,\mathrm{dist}_\Sigma(p,q).
\end{align*}
\end{proof}

\begin{cor}[Distance estimates] \label{thm:DistEst} Let $f \colon (\Sigma^k, g_\Sigma, \vol_\Sigma) \to (X^n, g_X, \phi)$ be a smooth map, $\phi \in \Omega^k(X)$, where $\Sigma$ is complete and $\mathrm{Ric}_\Sigma \geq -E$ for some $E > 0$.
\begin{enumerate}[(a)]
\item Suppose $f$ is $k$-harmonic and weakly conformal.  If $\mathrm{Scal}_\Sigma \geq -S$ for $S \geq 0$, and if $\mathrm{Sec}_X \leq -B$ for $B > 0$, then
$$\mathrm{dist}_X(f(p), f(q)) \leq \sqrt{ \frac{S}{k(k-1)B } }\,\mathrm{dist}_\Sigma(p,q), \ \ \forall p,q \in \Sigma.$$
\item Suppose $f$ is a Smith immersion.  If $\mathrm{Scal}_\Sigma \geq -S$ for $S \geq 0$, and if $\mathrm{Sec}^\phi_X \leq -A$ for $A > 0$, then
$$\mathrm{dist}_X(f(p), f(q)) \leq \sqrt{ \frac{S}{A} }\,\mathrm{dist}_\Sigma(p,q), \ \ \forall p,q \in \Sigma.$$
\end{enumerate}
\end{cor}

\begin{proof} ${}$
\begin{enumerate}[(a)]
\item The Schwarz Lemma \ref{thm:SchwarzMain}(a) gives $|df|^2 \leq \frac{S}{(k-1)B}$.  Now apply the basic distance estimate \ref{lem:BasicDistance}.
\item The Schwarz Lemma \ref{thm:SchwarzMain}(b) gives $|df|^2 \leq \frac{kS}{A}$.  Now apply the basic distance estimate \ref{lem:BasicDistance}.
\end{enumerate}
\end{proof}

\indent It is well-known that if a K\"{a}hler manifold $X$ has holomorphic sectional curvature bounded above by a negative constant, then $X$ is Brody hyperbolic.  We now generalize this fact to arbitrary calibrated geometries:

\begin{thm}[Negative $\phi$-curvature implies $\phi$-hyperbolicity] \label{cor:NegativeCurvature} Let $(X, g_X, \phi)$ be a calibrated manifold.  If $\mathrm{Sec}^\phi_X \leq -A$ for some $A > 0$, then $X$ is $\phi$-hyperbolic.
\end{thm}

\begin{proof} Suppose that $\mathrm{Sec}^\phi_X \leq -A$ for some $A > 0$.  If $f \colon (\R^k, g_0, \vol_0) \to (X^n, g_X, \phi)$ is a Smith immersion, Corollary \ref{thm:DistEst}(b) implies that $\mathrm{dist}_X(f(p), f(q)) = 0$ for all $p,q \in \R^k$, so $f$ is constant.
\end{proof}

\begin{cor} \label{cor:Vol-Hyperbolic} Let $(X^n, g_X, \vol_X)$ be an oriented Riemannian $n$-manifold.  If $\mathrm{Scal}_X \leq -A$ for some $A > 0$, then $X$ is $\vol_X$-hyperbolic (i.e., every orientation-preserving, weakly conformal map $(\R^n, g_0) \to (X^n, g_X)$ is constant).
\end{cor}

\begin{proof} This follows from Corollary \ref{cor:NegativeCurvature} and Example \ref{example:Sectional}.
\end{proof}

\indent In a first course in complex analysis, one typically learns the Schwarz lemma as a pointwise bound on holomorphic functions $f \colon B^2(R_1) \to B^2(R_2)$ satisfying $f(0) = 0$.  This is a special case of the following:

\begin{prop}[Classic Schwarz Lemma] \label{thm:ClassicSchwarz} Let $f \colon (B^k(R_1), g_{R_1}) \to (B^n(R_2), g_{R_2})$ be a weakly conformal, $k$-harmonic map.
\begin{enumerate}[(a)]
\item For all $w \in T_0B^k$ with $|w| = 1$, we have $|df_0(w)|_{R_2} \leq \frac{R_2}{R_1}$.
\item If $f(0) = 0$, then $\left| f(x) \right| \leq \frac{R_2}{R_1} |x|$ for all $x \in B^k(R_1)$.  (Here, $|\cdot| = |\cdot|_0$ is the euclidean norm.)
\end{enumerate}
\end{prop}
${}$
\begin{proof} Before beginning, we note that $\mathrm{Scal}_{B^k(R_1)} = -\frac{4}{R_1^2}k(k-1)$ and that $\mathrm{Sec}_{B^n(R_2)} = -\frac{4}{R_2^2}$.
\begin{enumerate}[(a)]
\item The Schwarz Lemma \ref{thm:SchwarzMain}(a) with $S = \frac{4}{R_1^2}k(k-1)$ and $B = \frac{4}{R_2^2}$ implies that $|df|^2 \leq k\frac{R_2^2}{R_1^2}$.  Now the basic distance estimate \ref{lem:BasicDistance} with $D = k\frac{R_2^2}{R_1^2}$ gives the result.
\item Let $x \in B^k(R_1)$.  The distance estimate of Corollary \ref{thm:DistEst}(a) with  $S = \frac{4}{R_1^2}k(k-1)$ and $B = \frac{4}{R_2^2}$ gives
$$\mathrm{dist}_{R_2}(f(x), f(0)) \leq \frac{R_2}{R_1}\,\mathrm{dist}_{R_1}(x,0).$$
Recalling the explicit formula for the Poincar\'{e} distance (Proposition \ref{prop:Poincare-Facts}(d)), and using that $f(0) = 0$, we obtain
$$\mathrm{arcsinh}\sqrt{ \frac{  | f(x) |^2 }{ R_2^2 - |f(x)|^2  } } \leq   \mathrm{arcsinh}\sqrt{ \frac{  | x |^2 }{ R_1^2 - |x|^2  } } .$$
Since $h(t) = \mathrm{arcsinh} (\sqrt{t})$ is an increasing function, it follows that
$$ \frac{  | f(x) |^2 }{ R_2^2 - |f(x)|^2  }  \leq    \frac{  | x |^2 }{ R_1^2 - |x|^2  }.$$
Rearranging yields $|f(x)| \leq \frac{R_2}{R_1}|x|$ as claimed.
\end{enumerate}
\end{proof}

\begin{cor}[Bounded domains in $\R^n$ are $\phi$-hyperbolic] \label{cor:BoundedDomains} Let $U \subset \R^n$ be an open set, and let $\phi \in \Omega^k(\R^n)$ be a calibration.  If $U$ is bounded, then $U$ is $\phi$-hyperbolic.
\end{cor}

\begin{proof} Suppose $U$ is bounded, so that $U \subset B^n(r)$ for some $r > 0$.  Let $f \colon (\R^k, g_0, \vol_0) \to (U, g_0, \phi)$ be a Smith immersion, and let $x_0 \in \R^k$ be arbitrary.  For each $R > |x_0|$, the map $h_R \colon B^k(R) \to B^n(2r)$ given by $h_R(x) = f(x) - f(0)$ is weakly-conformal and $k$-harmonic with $h_R(0) = 0$, so Proposition \ref{thm:ClassicSchwarz}(b) implies that
$$| f(x) - f(0) | =  | h_R(x) | \leq \frac{2r}{R}|x|, \text{ for all }x \in B^k(R).$$
In particular, $\left| f(x_0) - f(0) \right| \leq \frac{2r}{R}|x_0|$.  Letting $R \to \infty$, we see that $|f(x_0) - f(0)| = 0$, so $f(x_0) = f(0)$.  Since $x_0 \in \R^k$ was arbitrary, we conclude that $f$ is constant.
\end{proof}

\section{The KR $\phi$-metric}

\indent \indent Our next goal is to generalize the Kobayashi-Royden pseudo-metric $F_X$ (recall (\ref{eq:KRPM})) to arbitrary calibrated geometries.  In $\S$\ref{sub:KR}, for each calibrated manifold $(X,g_X, \phi)$, we define a Finsler pseudo-metric $\mathcal{K}_{(X,\phi)} \colon TX \to [0,\infty]$ called the \emph{KR $\phi$-metric}.  A fundamental fact about $\mathcal{K}_{(X,\phi)}$ is its ``decreasing property" with respect to conformally calibrating maps (Proposition \ref{prop:Dec1}). \\
\indent In practice, deriving an explicit formula for $\mathcal{K}_{(X,\phi)}$ is highly non-trivial.  Even in the K\"{a}hler case, the Kobayashi-Royden pseudo-metric $F_X = \mathcal{K}_{(X,\omega)}$ has been calculated in only a small handful of cases, such as complex hyperbolic space $X = \mathbb{CH}^n$.  The main results of this section (Theorems \ref{thm:K-Poincare} and \ref{thm:Quat-Hyp}) are explicit formulas for the KR $\vol_1$-metric of real hyperbolic space $\mathbb{RH}^n = (B^n, g_1)$ and the KR $\Psi$-metric of quaternionic hyperbolic space $\mathbb{HH}^n$.  The latter calculation makes crucial use of the Schwarz Lemma \ref{thm:SchwarzMain}(b).

\subsection{The KR $\phi$-metric} \label{sub:KR}

\indent \indent Let $(X^n, g_X, \phi)$ be a calibrated $n$-manifold, where $\phi \in \Omega^k(X)$.   Regard the unit $k$-ball $B^k := B^k(1)$ as the calibrated manifold $(B^k, g_1, \vol_1)$, where $g_1$ is the Poincar\'{e} metric and $\vol_1 \in \Omega^k(B^k)$ is the corresponding volume form as discussed in $\S$\ref{sub:Hyperbolic-Space}.  Let
$$\mathrm{SmIm}(B^k, X) = \{\text{Smith immersions } f\colon (B^k, g_1, \vol_1) \to (X^n, g_X, \phi)\}.$$

\begin{defn} The \emph{KR $\phi$-metric} on $(X, g_X)$ is the function
\begin{align*}
\mathcal{K}_{(X,\phi)} \colon TX & \to [0,  \infty] \\
\mathcal{K}_{(X,\phi)}(v_p) & = \inf\!\left\{ a > 0 \colon \exists f \in \mathrm{SmIm}(B^k, X), w \in T_0B^k, |w| = 1 \text{ s.t. } f(0) = p,\, df_0(w) = \frac{1}{a}v \right\}\!.
\end{align*}
If there does not exist a Smith immersion $f \colon B^k \to X$, a unit vector $w \in T_0B^k$, and an $a > 0$ such that $f(0) = p$ and $df_0(w) = \frac{1}{a}v$, then $\mathcal{K}_{(X,\phi)}(v) = \infty$.  We typically use the abbreviation $\mathcal{K}_\phi = \mathcal{K}_{(X,\phi)}$.  
\end{defn}

\indent  Geometrically, $\mathcal{K}_{(X,\phi)}(v)$ can be viewed as the reciprocal of the largest conformal factor that a Smith $k$-disk having $v \in T_pX$ as a tangent vector can have at the point $p$.  When $(X, g_X, \omega)$ is a K\"{a}hler manifold, the KR $\omega$-metric coincides with the Kobayashi-Royden pseudo-metric (\ref{eq:KRPM}).  As such, we now show that many familiar properties extend to arbitrary calibrated manifolds.

\begin{prop} \label{lem:K-Basics}  Let $(X,g_X, \phi)$ be a calibrated manifold, and let $v \in TX$.  Define
\begin{align*}
\mathcal{W}_\phi \colon \mathrm{Gr}(\phi) & \to [0,  \infty] \\
\mathcal{W}_\phi(\xi_p) & = \inf\!\left\{ \frac{1}{\Vert df_0 \Vert} \colon f \in \mathrm{SmIm}(B^k, X), f(0) = p,\, df_0(T_0B^k) = \xi\right\}\!.
\end{align*}
Then we have the formula
$$\mathcal{K}_\phi(v) =  \left| v \right|\, \inf\!\left\{ \mathcal{W}_\phi(\xi) \colon \xi \in \mathrm{Gr}(\phi) \text{ s.t. } v \in \xi \right\}\!.$$
Consequently, $\mathcal{K}_\phi(av) = \left|a\right| \mathcal{K}_\phi(v)$ for all $a \in \R$.
\end{prop}
\begin{proof} Fix $v \in T_pX$.  Let $\widehat{\mathcal{K}}_\phi(v) = \left| v \right|\, \inf\!\left\{ \mathcal{W}_\phi(\xi) \colon \xi \in \mathrm{Gr}(\phi) \text{ s.t. } v \in \xi \right\}$.  We aim to show that $\mathcal{K}_\phi(v) = \widehat{\mathcal{K}}_\phi(v)$.  \\
\indent Suppose first that $v = 0$.  Then clearly $\widehat{\mathcal{K}}_\phi(v) = 0$.  On the other hand, letting $f \colon (B^k, g_1) \to (X, g_X)$ denote the constant map $f(x) = p$, we see that every $w \in T_0B^k$ with $|w| = 1$ satisfies $df_0(w) = 0$, and so $\mathcal{K}_\phi(v) = 0$, too. \\
\indent Suppose now that $v \neq 0$ and $\mathcal{K}_\phi(v)$ is finite.  Fix $\epsilon > 0$.  Then there exists a Smith immersion $f \colon (B^k, g_1) \to (X,g_X)$ and $w \in T_0B^k$ with $|w| = 1$ such that $f(0) = p$ and $df_0(w) = \frac{1}{a}v$ with $0 < a < \mathcal{K}_\phi(v) + \epsilon$.  Note that since $f$ is conformal, we have $\Vert df_0 \Vert = \Vert df_0 \Vert |w| = |df_0(w)| = \frac{1}{a}|v|$.  Moreover, since $v \neq 0$, the map $df_0\colon T_0B^k \to T_pX$ is injective, so $\xi := df_0(T_0B^k) \in \mathrm{Gr}(\phi)$ and clearly $v \in \xi$.  Therefore,
$$\widehat{\mathcal{K}}_\phi(v) \leq |v|\, \mathcal{W}_\phi(\xi) \leq |v| \frac{1}{\Vert df_0 \Vert} = a < \mathcal{K}_\phi(v) + \epsilon.$$
Letting $\epsilon \to 0$ shows that $\widehat{\mathcal{K}}_\phi(v) \leq \mathcal{K}_\phi(v)$. \\
\indent For the other direction, again assuming $ \widehat{\mathcal{K}}_\phi(v)$ is finite, fix $\epsilon > 0$.  Then there exists $\xi \in \mathrm{Gr}(\phi)$ with $v \in \xi$ such that $|v|\, \mathcal{W}_\phi(\xi) < \widehat{\mathcal{K}}_\phi(v) + \frac{1}{2}\epsilon$.  Moreover, there exists a Smith immersion $f \colon (B^k, g_1) \to (X, g_X)$ having $f(0) = p$ and $df_0(T_0B^k) = \xi$ such that $\frac{1}{\Vert df_0 \Vert} < \mathcal{W}_\phi(\xi) + \frac{1}{2|v|}\epsilon$.  Since $v \in \xi$, there exists $w \in T_0B^k$ with $|w| = 1$ and $a > 0$ such that $df_0(w) = \frac{1}{a}v$.  Again since $f$ is conformal, we have $\Vert df_0 \Vert = \frac{1}{a}|v|$.  Therefore,
$$\mathcal{K}_\phi(v) \leq a = |v| \frac{1}{\Vert df_0 \Vert} < |v| \mathcal{W}_\phi(\xi) + \frac{1}{2}\epsilon < \widehat{\mathcal{K}_\phi}(v) + \epsilon.$$
Letting $\epsilon \to 0$ gives $\mathcal{K}_\phi(v) \leq \widehat{\mathcal{K}}_\phi(v)$.  In fact, we have shown that  $\mathcal{K}_\phi(v)$ is finite if and only if $\widehat{\mathcal{K}}_\phi(v)$ is finite.  Consequently, $\mathcal{K}_\phi(v) = \infty$ if and only if $\widehat{\mathcal{K}}_\phi(v) = \infty$.   \\
\indent Finally, the last claim follows from observing that $\mathcal{K}_\phi(av) = \widehat{\mathcal{K}}_\phi(av) = \left|a\right| \widehat{\mathcal{K}}_\phi(v) = |a|\,\mathcal{K}_\phi(v)$.
\end{proof}

\begin{prop}[Decreasing property] \label{prop:Dec1} Let $f \colon (X^m, g_X, \alpha) \to (Y^n, g_Y, \beta)$ be a conformally calibrating map.  Then for all $p \in X$ and $v \in T_pX$, we have
$$\mathcal{K}_{(Y,\beta)}(df_p(v)) \leq \mathcal{K}_{(X, \alpha)}(v).$$
\end{prop}

\begin{proof} Fix $p \in X$ and $v \in T_pX$.  Consider the sets
\begin{align*}
S_X & = \left\{ a > 0 \colon \exists F \in \mathrm{SmIm}(B^k, X^m), w \in T_0B^k, |w| = 1 \text{ s.t. } F(0) = p,\, dF_0(w) = \frac{1}{a}v\right\} \\
S_Y & = \left\{ b > 0 \colon \exists G \in \mathrm{SmIm}(B^k, Y^n), w \in T_0B^k, |w| = 1  \text{ s.t. } G(0) = f(p),\, dG_0(w) = \frac{1}{b}df_p(v)\right\}\!.
\end{align*}
We claim that $S_X \subset S_Y$.  For this, let $a \in S_X$, so there exists a Smith immersion $F \colon (B^k, g_1, \vol_1) \to (X^m, g_X, \alpha)$ and $w \in T_0B^k$ with $|w| = 1$ such that $F(0) = p$ and $dF_0(w) = \frac{1}{a}v$.  Setting $G := f \circ F \colon (B^k, g_1, \vol_1) \to (Y^n, g_Y, \beta)$, we observe that $G$ is a Smith immersion having $G(0) = f(F(0)) = f(p)$ and $dG_0(w) = df_p( dF_0(w) ) = \frac{1}{a} df_p(v)$.  This shows that $a \in S_Y$, which proves the claim.  Finally, since $S_X \subset S_Y$, it follows that $\mathcal{K}_{(Y, \beta)}(df_p(v)) = \inf(S_Y) \leq \inf(S_X) = \mathcal{K}_{(X, \alpha)}(v)$.
\end{proof}

\begin{cor}[Comparison principle] \label{cor:Comparison-K} Let $(X, g_X, \phi)$ be a calibrated manifold, and let $U \subset X$ be an open set.  Then $\mathcal{K}_{(X, \phi)}(v) \leq \mathcal{K}_{(U, \phi)}(v)$ for all $v \in TU$.
\end{cor}

\begin{proof} Let $\iota \colon (U, g_X, \phi) \hookrightarrow (X, g_X, \phi)$ be the inclusion map.  Then $\iota$ is a conformally calibrating map.  So, for all $v \in TU$, Proposition \ref{prop:Dec1} gives $ \mathcal{K}_{(X, \phi)}(v) =  \mathcal{K}_{(X, \phi)}( d\iota_p(v) ) \leq \mathcal{K}_{(U, \phi)}(v)$.
\end{proof}

\begin{cor}[Invariance] \label{cor:Invariance} Let $(X, g_X, \phi)$, $(Y, g_Y, \psi)$ be a calibrated manifolds.  If $F \colon X \to Y$ is a Smith equivalence, then for all $p \in X$ and $v \in T_pX$, we have:
$$\mathcal{K}_{(Y,\psi)}(dF_p(v)) = \mathcal{K}_{(X,\phi)}(v).$$
\end{cor}

\begin{proof} Let $F \colon (X, g_X, \phi) \to (Y, g_Y, \psi)$ be a Smith equivalence.  Fix $p \in X$ and $v \in T_pX$.  Set $q = F(p)$.  Using Proposition \ref{prop:Dec1} twice, observing that $F^{-1}$ is also a Smith equivalence, we have
$$\mathcal{K}_{(Y,\psi)}(dF_p(v)) \leq \mathcal{K}_{(X,\phi)}(v) = \mathcal{K}_{(X,\phi)}( dF^{-1}_q (dF_p(v) ) ) \leq \mathcal{K}_{(Y,\psi)}( dF_p(v)).$$
\end{proof}

\subsection{Examples} \label{sub:Examples-KR}

\begin{example}[K\"{a}hler calibration] \label{ex:KobRoyden} In general, an explicit formula for the Kobayashi-Royden pseudometric of a complex manifold is difficult to find.  Let $X,Y$ be complex manifolds.
\begin{itemize}
\item On the unit disk $\mathbb{D} = \{ z \in \mathbb{C} : | z|<1 \}$, the Kobayashi-Royden metric is given by the (square root of the) Poincar\'e metric: $\mathcal{K}_{(\mathbb{D}, \omega)}(v_z) = \frac{|v|}{1-|z|^2}$.  More generally, on the unit ball $B^{2n} := \{ z \in \C^n \colon \sum_k | z_k |^2 <1 \}$, the Kobayashi--Royden metric is given by
$$\mathcal{K}_{(B^{2n}, \omega)}(v_z) = \frac{\sqrt{(1-|z|^2)|v|^2 + | \langle z,v \rangle |^2}}{1-|z|^2}$$
where $\langle \cdot, \cdot \rangle$ and $|\cdot|$ denote the standard Hermitian inner product on $\C^n$.
\item For convex domains, which includes the Hermitian symmetric spaces of non-compact type, the Kobayashi-Royden metric coincides with the Carath\'eodory-Reiffen metric \cite{lempert1981metrique}.
\item We have the following formula for products: $\mathcal{K}_{X \times Y}(u,v) = \max \{ \mathcal{K}_X(u), \mathcal{K}_Y(v)\}$ for $u \in TX, v \in TY$ (see, e.g., \cite[Proposition 3.5.25]{kobayashi2013hyperbolic}).
\item If $\pi \colon \widetilde{X} \to X$ is a covering space, then $\mathcal{K}_{( \widetilde{X}, \widetilde{\omega}) } = \pi^{\ast} \mathcal{K}_{(X,\omega)}$ (see, e.g., \cite[Proposition 3.5.26]{kobayashi2013hyperbolic}).
\end{itemize}
\end{example}

\begin{prop}[Euclidean balls] \label{ex:FlatElliptic} Let $\phi \in \Lambda^k(\R^n)^*$ be a constant-coefficient elliptic calibration.
\begin{enumerate}[(a)]
\item Consider the calibrated manifold $(B(p;r), g_0, \phi)$, where $B(p; r) = \{ x \in \R^n \colon |x-p| < r\}$.  For all $y \in B(p;r)$ and $v \in T_yB(p;r)$, we have $\mathcal{K}_{(B(p;r), \phi)}(v) \leq \frac{1}{r - |y-p|}|v|$.
\item Consider the calibrated manifold $(\R^n, g_0, \phi)$.  We have $\mathcal{K}_{(\R^n, \phi)}(v) = 0$ for all $v \in T\R^n$.
\end{enumerate}
\end{prop}

\begin{proof} (a) Let $y \in B(p;r)$ and $v \in T_y\R^n$.  If $v = 0$, the claim is clear, so assume $v \neq 0$.  Since $\phi$ is elliptic, there exists a $\phi$-calibrated plane $\Pi \subset T_y\R^n$ with $v \in \Pi$.  Identifying $T_y\R^n \cong \R^n$, we regard $\Pi$ as an affine $\phi$-plane through $y$ containing the vector $v$, and we let $\{\frac{v}{|v|}, v_2, \ldots, v_k\}$ be an orthonormal basis for $\Pi$.  Now, for each $a > \frac{|v|}{r - |y-p|}$, let $f_a \colon (\R^k, g_0, \vol_0) \to (\R^n, g_0, \phi)$ be
$$f_a(x_1, \ldots, x_k) = y + \frac{|v|}{a}\left( x_1 \frac{v}{|v|} + \sum_{j=2}^k x_j v_j \right)\!.$$
Note that $f_a$ is a Smith immersion whose image is $\Pi$. \\
\indent Now, let $\iota \colon (B^k, g_1, \vol_1) \to (\R^k, g_0, \vol_0)$ be the inclusion map, and consider the map $F_a := f_a \circ \iota \colon (B^k, g_1, \vol_1) \to (\R^n, g_0, \phi)$.  Note that $F_a$ is a Smith immersion with $F_a(0) = y$ and $(dF_a)|_0(e_1) = \frac{1}{a}v$, and $\mathrm{Image}(F_a) \subset B(y; \frac{|v|}{a}) \subset B(y; r - |y-p|) \subset B(p;r)$.  This shows that $\mathcal{K}_{(B(p;r), \phi)}(v) \leq \frac{1}{r-|y-p|}|v|$. \\
\indent (b) The same argument applies, but now there is no restriction on $a > 0$.
\end{proof}

\begin{example}[Volume form calibration] \label{ex:Vol-Form-KR} Let $(X, g_X, \vol_X)$ be a calibrated $n$-manifold equipped with the volume form calibration $\vol_X \in \Omega^n(X)$.  If $(X, g_X)$ is not locally conformally flat, then $\mathcal{K}_{(X, \vol_X)}(v) = \infty$ for all $v \in TX$, $v \neq 0$.
\end{example}

\begin{proof} We prove the contrapositive.  Let $p \in X$ and $v \in T_pX$, $v \neq 0$ be such that $\mathcal{K}_{(X, \vol_X)}(v) < \infty$.  Then there exists a Smith immersion $f \colon (B^n, g_1, \vol_1) \to (X^n, g_X, \vol_X)$, a vector $w \in T_0B^n$ with $|w| = 1$, and $a > 0$ such that $f(0) = p$ and $df_0(w) = \frac{1}{a}v$.  Noting that $f$ is strict at $0 \in B^n$, there exists a neighborhood $U \subset B^n$ of $0$ for which $f|_U \colon (U, g_1, \vol_1) \to (X^n, g_X, \vol_X)$ is a conformal diffeomorphism.  Thus, $(X,g_X)$ is locally conformally flat.
\end{proof}

\indent We now compute the KR $\vol_1$-metric of real hyperbolic space $(B^n, g_1)$.

\begin{thm} \label{thm:K-Poincare} Let $(B^n, g_1)$ be the unit $n$-ball with its Poincar\'{e} metric, and equip it with the volume form calibration $\vol_1 \in \Omega^n(B^n)$.  Then for all $p \in B^n$ and $v \in T_pB^n$, we have
$$\mathcal{K}_{(B^n, \vol_1)}(v) = |v|_1.$$
\end{thm}

\begin{proof} Suppose first that $p = 0$, and fix $v \in T_0B^n$.  If $v = 0$, the result is clear, so assume $v \neq 0$.  Notice that the identity map $\mathrm{Id} \colon (B^n, g_1, \vol_1) \to (B^n, g_1, \vol_1)$ is a Smith immersion with $\mathrm{Id}(0) = 0$ and $d\,\mathrm{Id}_0(\frac{v}{|v|_1}) = \frac{1}{|v|_1}v$.  This shows that $\mathcal{K}_{(B^n, \vol_1)}(v) \leq |v|_1$. \\
\indent In the other direction, for each $a > 0$, let $f \colon (B^n, g_1, \vol_1) \to (B^n, g_1, \vol_1)$ be a Smith immersion, and let $w \in T_0B^k$ with $|w| = 1$ be such that $f(0) = 0$ and $df_0(w) = \frac{1}{a}v$.  By the classic Schwarz lemma (Proposition \ref{thm:ClassicSchwarz}(a)), we have $\frac{1}{a}|v|_1 = |df_0(w)|_1 \leq 1$, so that $|v|_1 \leq a$.  Therefore, we deduce that $|v|_1 \leq \mathcal{K}_{(B^n, \vol_1)}(v)$.  This proves the result at the origin $p = 0$. \\
\indent Finally, let $p \in B^n$ be arbitrary, and let $v \in T_pB^n$.  The Smith automorphisms of $(g_1, \vol_1)$ coincide with the orientation-preserving conformal automorphisms of $g_1$.  Since this group acts transitively on $B^n$ (Proposition \ref{prop:Poincare-Facts}), there exists a Smith automorphism $F \colon (B^n, g_1) \to (B^n, g_1)$ such that $F(p) = 0$.  Now, using invariance (Corollary \ref{cor:Invariance}), followed by the result for $p = 0$, followed by the fact that conformal automorphisms of $B^n$ are isometries for the Poincar\'{e} metric (Proposition \ref{prop:Poincare-Facts}), we have
$$\mathcal{K}_{(B^n, \vol_1)}(v) = \mathcal{K}_{(B^n, \vol_1)}(dF_p(v)) = \left| dF_p(v) \right|_1 = \left| v \right|_1.$$
\end{proof}

\begin{cor} \label{cor:Dec2} Let $f \colon (B^k, g_1, \vol_1) \to (X^n, g_X, \phi)$ be a Smith immersion, $\phi \in \Omega^k(X)$.  For all $p \in B^k$ and $v \in T_pB^k$, we have
$$\mathcal{K}_{(X,\phi)}(df_p(v)) \leq |v|_1.$$
\end{cor}

\begin{proof} Combine Proposition \ref{prop:Dec1} and Theorem \ref{thm:K-Poincare}.
\end{proof}

\indent We now calculate the KR $\Psi$-metric of quaternionic hyperbolic space $\mathbb{HH}^n$, where $\Psi$ is the QK $4$-form.  Our convention is that $\mathbb{HH}^n$ has sectional curvature between $-4$ and $-1$.

\begin{thm} \label{thm:Quat-Hyp} Let $(\mathbb{HH}^n, g_{\mathbb{HH}^n})$ be quaternionic hyperbolic space, and let $\Psi \in \Omega^4(\mathbb{HH}^n)$ be the quaternionic-K\"{a}hler calibration.  Then for all $p \in \mathbb{HH}^n$ and $v \in T_p\mathbb{HH}^n$, we have
$$\mathcal{K}_{(\mathbb{HH}^n, \Psi)}(v) = \left|v\right|_{\mathbb{HH}^n},$$
where $|v|_{\mathbb{HH}^n} = \sqrt{ g_{\mathbb{HH}^n}(v,v)}$ is the norm arising from the Riemannian metric.
\end{thm}

\begin{proof} Suppose first that $p = 0$, and fix $v \in T_0\mathbb{HH}^n$.  If $v = 0$, the result is clear, so assume $v \neq 0$. \\
\indent Let $\iota \colon \mathbb{HH}^1 \to (\mathbb{HH}^n, g_{\mathbb{HH}^n}, \Psi)$ be the standard inclusion of a quaternionic hyperbolic line through $0$.  Note that the image is isometric to $(B^4(1), g_1)$ and is $\Psi$-calibrated.  Therefore, equipping $\mathbb{HH}^1$ with the pullback metric and $4$-form $(\iota^*g_{\mathbb{HH}^n}, \iota^*\Psi) = (g_{\mathbb{HH}^1}, \vol_{\mathbb{HH}^1}) \cong (g_1, \vol_1)$ makes $\iota$ into an isometric Smith immersion.  Moreover, with our choice of metric conventions, we have $|v|_{\mathbb{HH}^n} = |v|_1$ for vectors $v \in T_0\mathbb{HH}^n \cong T_0B^4$, and so $d\iota_0(\frac{v}{|v|_1}) = \frac{1}{|v|_{\mathbb{HH}^n}}v$.  This shows that $\mathcal{K}_{(\mathbb{HH}^n, \Psi)}(v) \leq |v|_{\mathbb{HH}^n}$. \\
\indent In the other direction, for $a > 0$, let $f \colon (B^4, g_1, \vol_1) \to (\mathbb{HH}^n, g_{\mathbb{HH}^n}, \Psi)$ be a Smith immersion, and let $w \in T_0B^4$ with $|w| = 1$ be such that $f(0) = 0$ and $df_0(w) = \frac{1}{a}v$.  Note that $S := -\mathrm{Scal}_{(B^4, g_1)} = 48$ and $A := -\mathrm{Sec}^\Psi_{\mathbb{HH}^n} = 48$.  Therefore, the Schwarz Lemma of Theorem \ref{thm:SchwarzMain}(b) implies that $|df|^2 \leq 4$, so the basic distance estimate \ref{lem:BasicDistance} gives $\frac{1}{a}|v|_{\mathbb{HH}^n} = |df_0(w)| \leq 1$.  Thus, we deduce that $|v|_{\mathbb{HH}^n} \leq \mathcal{K}_{(\mathbb{HH}^n, \Psi)}(v)$.  This proves the result at $p = 0$. \\
\indent Finally, let $p \in \mathbb{HH}^n$ be arbitrary, and let $v \in T_p\mathbb{HH}^n$.  Recalling that $\mathbb{HH}^n = \Sp(n,1)/(\Sp(n) \times \Sp(1))$, there exists $A \in \mathrm{Sp}(n,1)$ such that $A \cdot p = 0$, where $\cdot$ denotes the $\Sp(n,1)$-action on $\mathbb{HH}^n$.  The map $F \colon \mathbb{HH}^n \to \mathbb{HH}^n$ given by $F(x) = A \cdot x$ satisfies $F(p) = 0$ and is a Gray automorphism (i.e., an isometric Smith automorphism).  Thus, using invariance (Corollary \ref{cor:Invariance}), followed by the result for $p = 0$, followed by the fact that Gray automorphisms are isometries, we have
$$\mathcal{K}_{(\mathbb{HH}^n, \Psi)}(v) = \mathcal{K}_{(\mathbb{HH}^n, \Psi)}(dF_p(v)) = \left| dF_p(v) \right|_{\mathbb{HH}^n} = |v|_{\mathbb{HH}^n}.$$
\end{proof}

\section{$R_\phi$-hyperbolicity and $K_\phi$-hyperbolicity} \label{sec:RphiKphi}

\indent \indent In this section, we will use the KR $\phi$-metric to define two notions of hyperbolicity, both of which generalize Kobayashi hyperbolicity.  The first, called $R_\phi$-hyperbolicity, is defined by a lower bound on the KR $\phi$-metric.  In $\S$\ref{sub:Rphi}, we will prove that strongly negative $\phi$-sectional curvature implies $R_\phi$-hyperbolicity, which in turn implies $\phi$-hyperbolicity (Theorem \ref{thm:RImpliesPhi}).  We will also provide examples of $\phi$-hyperbolic domains that are not $R_\phi$-hyperbolic (Theorem \ref{thm:CounterConverse}). \\
\indent In $\S$\ref{sub:KobDist}, we extend the Kobayashi pseudo-distance to calibrated manifolds $(X, g_X, \phi)$ having the property that its KR $\phi$-metric $\mathcal{K}_{(X,\phi)} \colon TX \to [0,\infty]$ is finite and upper semicontinuous.  We shall say that $X$ is $K_\phi$-hyperbolic if the pseudo-distance obtained from integrating $\mathcal{K}_{(X,\phi)}$ is a genuine distance function.  In Theorem \ref{thm:K-implies-Phi}, we shall prove that (under the semicontinuity assumption) $R_\phi$-hyperbolicity implies $K_\phi$-hyperbolicity, and that $K_\phi$-hyperbolicity implies $\phi$-hyperbolicity.

\subsection{$R_\phi$-hyperbolicity} \label{sub:Rphi}

\begin{defn} Let $(X, g_X, \phi)$ be a calibrated manifold.  Say $(X, g_X)$ is \emph{$R_\phi$-hyperbolic} if for each $p \in X$, there exists a neighborhood $U \subset X$ of $p$ and a constant $c > 0$ such that $\mathcal{K}_{(X,\phi)}(v) \geq c \left|v\right|_X$ for all $v \in TU$.
\end{defn}

\indent When $(X, g, \omega)$ is a K\"{a}hler manifold, Royden \cite{royden2006remarks} proved that $R_\omega$-hyperbolicity is equivalent to Kobayashi hyperbolicity.

\begin{prop} Let $(X, g_X, \phi)$ be a calibrated manifold, $\phi \in \Omega^k(X)$.
\begin{enumerate}[(a)]
\item Suppose $(X, g_X, \phi)$ is Smith equivalent to $(Y, g_Y, \psi)$.  Then $(X, g_X)$ is $R_\phi$-hyperbolic if and only if $(Y, g_Y)$ is $R_\psi$-hyperbolic.
\item Let $U \subset X$ be an open subset.  If $(X, g_X)$ is $R_\phi$-hyperbolic, then $(U, g_X|_U)$ is $R_\phi$-hyperbolic.
\end{enumerate}
\end{prop}

\begin{proof} ${}$
\begin{enumerate}[(a)]
\item Suppose $(X, g_X)$ is $R_\phi$-hyperbolic.  Let $F \colon (X, g_X, \phi) \to (Y, g_Y, \psi)$ be a Smith equivalence, and let $\lambda \colon X \to \R^+$ be the conformal factor.  Fix $p \in Y$.  By $R_\phi$-hyperbolicity, there is a neighborhood $U \subset X$ of $F^{-1}(p)$ and $c > 0$ such that $\mathcal{K}_{(X,\phi)}(w) \geq c |w|_X$ for all $w \in TU$.  By shrinking $U$ if necessary, we can assume that $\overline{U}$ is compact, so that $\sup_{\overline{U}} \lambda $ is finite.  Now, for all $v \in T(F(U))$, we have
\begin{align*}
\mathcal{K}_{(Y, \psi)}(v) = \mathcal{K}_{(X,\phi)}(dF^{-1}(v)) \geq c \left| dF^{-1}(v) \right|_X = \frac{c}{\lambda} \left| v \right|_Y \geq \frac{c}{\sup_{\overline{U}} \lambda} \left| v \right|_Y.
\end{align*}
\item This follows immediately from the Comparison Principle (Corollary \ref{cor:Comparison-K}).
\end{enumerate}
\end{proof}


\indent In the K\"{a}hler case, it is well-known that strongly negative holomorphic sectional curvature implies Kobayashi hyperbolicity, and that Kobayashi hyperbolicity implies Brody hyperbolicity.  We now show that both of these implications hold in general:

\begin{thm} \label{thm:RImpliesPhi} Let $(X, g_X, \phi)$ be a calibrated manifold.
\begin{enumerate}[(a)]
\item If $\mathrm{Sec}^\phi_X \leq -A$ for some $A > 0$, then $(X, g_X)$ is $R_\phi$-hyperbolic.
\item If $(X, g_X)$ is $R_\phi$-hyperbolic, then $(X, g_X)$ is $\phi$-hyperbolic.
\end{enumerate}
\end{thm}

\begin{proof} ${}$
\begin{enumerate}[(a)]
\item Suppose $\mathrm{Sec}^\phi_X \leq -A$ for some $A > 0$.  Let $p \in X$, let $v \in T_pX$ have $v \neq 0$, and let $\xi \in \mathrm{Gr}(\phi)$ have $v \in \xi$.  By the Schwarz Lemma of Theorem \ref{thm:SchwarzMain}(b), every Smith immersion $f \colon (B^k, g_1, \vol_1) \to (X, g_X, \phi)$ with $f(0) = p$ and $df_0(T_0B^k) = \xi$ satisfies $\Vert df_0 \Vert^2 = \frac{1}{k}|df_0|^2 \leq \frac{4k(k-1)}{A}$, and hence $\frac{1}{\Vert df_0\Vert } \geq \sqrt{\frac{A}{4k(k-1)}}$.   Consequently, $\mathcal{K}_{(X,\phi)}(v) \geq \sqrt{ \frac{A}{4k(k-1)} }\, |v|_X$, and thus $X$ is $R_\phi$-hyperbolic.
\item Suppose $X$ is $R_\phi$-hyperbolic.  Let $f \colon (\R^k, g_0, \vol_0) \to (X, g_X, \phi)$ be a Smith immersion.  Fix $p \in \R^k$ and $v \in T_p \R^k$.  By assumption, there exists a constant $c > 0$ such that $\mathcal{K}_{(X,\phi)}(df_p(v)) \geq c |df_p(v)|_X$.  On the other hand, by the decreasing property (Proposition \ref{prop:Dec1}), we have
$$c |df_p(v)|_X \leq \mathcal{K}_{(X, \phi)}(df_p(v)) \leq \mathcal{K}_{(\R^k, \vol_0)}(v) = 0,$$
where in the last step we used Example \ref{ex:FlatElliptic}(b).  Thus, $|df_p(v)|_X = 0$, which implies that $df_p(v) = 0$.  Thus, $f$ is constant, so $X$ is $\phi$-hyperbolic.
\end{enumerate}
\end{proof}

\begin{cor}[Hyperbolic spaces are $R_\phi$-hyperbolic] \label{prop:HypSpacesRphi} ${}$
\begin{enumerate}[(a)]
\item Real hyperbolic space $(B^n, g_1)$ is $R_{\vol_1}$-hyperbolic.
\item Complex hyperbolic space $(\mathbb{CH}^n, g_{\mathbb{CH}^n})$ is $R_{\omega}$-hyperbolic, where $\omega \in \Omega^2(\mathbb{CH}^n)$ is the K\"{a}hler calibration.
\item Quaternionic hyperbolic space $(\mathbb{HH}^n, g_{\mathbb{HH}^n})$ is $R_{\Psi}$-hyperbolic, where $\Psi \in \Omega^4(\mathbb{HH}^n)$ is the QK calibration.
\end{enumerate}
\end{cor}

\indent We now show that the converse of Theorem \ref{thm:RImpliesPhi}(b) is false in general.  In fact, for constant coefficient elliptic calibrations $\phi \in \Lambda^k(\R^n)^*$ that are inner M\"{o}bius rigid, we will exhibit a domain $U \subset \R^n$ that is $\phi$-hyperbolic, but not $R_\phi$-hyperbolic.

\begin{thm} \label{thm:CounterConverse} Let $\phi \in \Lambda^k(\R^n)^*$ be an inner M\"{o}bius rigid, constant-coefficient, elliptic calibration with $k \geq 2$.  There exists a domain $U \subset \R^n$ such that $(U,g_0)$ is $\phi$-hyperbolic, but not $R_\phi$-hyperbolic.
\end{thm}

\begin{proof}  Let $E \subset \R^n$ be a $\phi$-calibrated $k$-plane that contains the origin.  Choose a basis $(e_1, \ldots, e_k$, $u_1, \ldots, u_{n-k})$ of $\R^n$ such that $E = \mathrm{span}\{e_1, \ldots, e_k\}$, and let $(x_1, \ldots, x_k, y_1, \ldots, y_{n-k})$ denote the corresponding linear coordinates on $\R^k \times \R^{n-k}$.  Consider the domain
$$U = \{ (x,y) \in \R^k \times \R^{n-k} \colon (\text{if }y \neq 0, \,\text{then } |y| < 1 \text{ and } |x| |y| < 1) \text{ or } (\text{if }y = 0, \,\text{then } |x| < 1) \}.$$
We claim:
\begin{enumerate}[(a)]
\item There exists a sequence $p_j \in U$ and $v_j \in T_{p_j}U$ having $p_j \to (0,0)$ and $\mathcal{K}_\phi(v_j) < \frac{1}{j}|v_j|$.  Consequently, $U$ is not $R_\phi$-hyperbolic.
\item The domain $U$ contains no affine lines.  Consequently, $U$ is $\phi$-hyperbolic.
\end{enumerate}
Proofs of claims:
\begin{enumerate}[(a)]
\item For $j \geq 2$, let $p_j = \frac{1}{j}u_1$, and observe that $p_j \to (0,0)$.  Let $v_j = e_1|_{p_j} \in T_{p_j}U$, a sequence of unit vectors in the $e_1$-direction.  Consider the sequence of (linear) Smith immersions $f_j \colon (B^k, g_1, \vol_1) \to (U, g_0, \phi)$ defined by $f_j(z) = j(z,0) + p_j$.  The image of $f_j$ is the $k$-ball $\{(x,0) \colon |x| < j\} + p_j$, so $f_j$ is well-defined (i.e., its image lies in $U$).  Since $f_j(0) = p_j$ and $df_j|_0(e_1) = jv_j$, it follows that $\mathcal{K}_\phi(v_j) \leq \frac{1}{j} = \frac{1}{j}|v_j|$.  This proves claim (a).
\item Suppose for contradiction that there exists an affine line $L \subset U$, say $L =  \{(x_0 + tu, y_0 + tv) \colon t \in \R\}$
with $(x_0, y_0) \in U$ and $(u,v) \in E \oplus E^\perp = \R^n$ where $(u,v) \neq (0,0)$.  There are three cases to consider.
\begin{itemize}
\item Suppose $v \neq 0$.  Then for $|t|$ sufficiently large, we have $|y_0 + tv| \geq |t| |v| - |y_0| \geq 1$, showing that $L \not\subset \{ |y| < 1\}$, contradicting $L \subset U \subset \{|y| < 1\}$.
\item Suppose $v = 0$ and $y_0 \neq 0$.  Then for $|t|$ sufficiently large, we have $|x_0 + tu| |y_0 + tv| \geq  (|t| |u| - |x_0|)|y_0| \geq 1$, showing that $L \not\subset \{|x| |y| < 1\}$, contradicting $L \subset U \subset \{|x| |y| < 1\}$.
\item Suppose $v = 0$ and $y_0 = 0$.  Then $L$ is contained in the $k$-ball $\{(x,y) \in U \colon y = 0\} = \{(x,0) \in \R^k \times \R^{n-k} \colon |x| < 1\}$, which is clearly impossible.
\end{itemize}
Finally, since $U$ contains no affine lines, it cannot contain any affine $\phi$-planes.  Thus, by Theorem \ref{prop:ConfRigid}(b), $U$ is $\phi$-hyperbolic.
\end{enumerate}
\end{proof}

\subsection{The Kobayashi $\phi$-Pseudo-distance} \label{sub:KobDist}

\indent \indent We now aim to define the Kobayashi $\phi$-pseudo-distance.  For this, we need to place a regularity assumption on the KR $\phi$-metric.  Recall that for a calibrated manifold $(X, g_X, \phi)$, its KR $\phi$-pseudo-metric is
\begin{align*}
\mathcal{K}_{(X, \phi)} \colon TX & \to [0,  \infty] \\
\mathcal{K}_{(X, \phi)}(v_p) & = \inf\!\left\{ a > 0 \colon \exists f \in \mathrm{SmIm}(B^k, X), w \in T_0B^k, |w| = 1 \text{ s.t. } f(0) = p,\, df_0(w) = \frac{1}{a}v \right\}\!.
\end{align*}

\begin{defn} We say that $(X,g_X)$ is \emph{$\phi$-replete} if the KR $\phi$-metric $\mathcal{K}_{(X,\phi)} \colon TX \to [0,\infty]$ is finite and upper semicontinuous.  
\end{defn}

\begin{example} ${}$
\begin{itemize}
\item If $\phi \in \Omega^k(\R^n)$ is a constant-coefficient elliptic calibration, then $(\R^n, \phi)$ is $\phi$-replete.
\item Every K\"{a}hler manifold $(X^{2m}, g_X, \omega)$ is $\omega$-replete \cite{royden2006remarks}.  Note, however, that $\mathcal{K}_{(X,\omega)}$ need not be continuous \cite[Example 3.5.37]{kobayashi2013hyperbolic}.
\item The Poincar\'{e} ball $(B^n, g_1)$ is $\vol_1$-replete (by Theorem \ref{thm:K-Poincare}).
\item Quaternionic hyperbolic space $(\mathbb{HH}^n, g_{\mathbb{HH}^n})$ is $\Psi$-replete, where $\Psi$ is the QK calibration  (by Theorem \ref{thm:Quat-Hyp}).
\end{itemize}
\end{example}

\begin{example}[Non-repleteness] Let $U \subset \R^n$ be a domain, and let $\phi = dx^1 \wedge \cdots \wedge dx^k \in \Omega^k(U)$ for $k < n$.  (Note that $\phi$ is \emph{not} an elliptic calibration.)  Then $(U, g_0)$ is \emph{not} $\phi$-replete.  Indeed, let $p \in U$, let $v \in T_pU$, and write $v = v_1 + v_2$ where $v_1 \in \mathrm{span}\{ \frac{\partial}{\partial x^1}, \ldots, \frac{\partial}{\partial x^k}\}$ and $v_2 \in \mathrm{span}\{ \frac{\partial}{\partial x^{k+1} }, \ldots, \frac{\partial}{\partial x^n}\}$.  If $v_2 = 0$, then $\mathcal{K}_{(U,\phi)}(v)$ is finite.  However, if $v_2 \neq 0$, then $\mathcal{K}_{(U,\phi)}(v) = \infty$.
\end{example}

\begin{rmk}[Almost complex case] Let $(M, J)$ be an almost-complex manifold.  In this generality, the Kobayashi-Royden pseudo-metric $F_X \colon TM \to [0,\infty]$ still makes sense.  If $J$ is class $C^{0,\alpha}$, $\alpha > 0$, then for every $v \in TM$, there exists a germ of $J$-holomorphic disks that have $v$ as a tangent vector (see \cite[Theorem III]{nijenhuis1963some}, \cite[Appendix 1]{ivashkovich2004schwarz}).  The existence of such disks ensures that $F_X(v)$ is finite, and a reasonably large supply ensures that $F_X$ is upper semicontinuous. \\
\indent When $J$ is class $C^{1,\alpha}$, the upper semicontinuity of $F_X$ was established in \cite{ivashkovich2004schwarz}.  In \cite{ivashkovich2005upper}, an almost-complex structure $J$ of class $C^{0,1/2}$ on the polydisk $\mathbb{D} \times \mathbb{D}$ was constructed such that $F_{\mathbb{D} \times \mathbb{D}}$ is not upper semicontinuous.
\end{rmk}

\begin{defn} Suppose that $(X, g_X)$ is $\phi$-replete.  The \emph{Kobayashi $\phi$-pseudo-distance} is
\begin{align*}
d_\phi = d_{(X,\phi)} \colon X \times X & \to [0,\infty] \\
d_\phi(p,q) = d_{(X,\phi)}(p,q) & = \inf_\gamma \int_0^1 \mathcal{K}_{(X,\phi)}( \gamma'(t))\,dt
\end{align*}
where the infimum is over all piecewise-smooth paths $\gamma \colon [0,1] \to X$ with $\gamma(0) = p$ and $\gamma(1) = q$.  It is straightforward to check that $d_{(X, \phi)}$ really is a pseudo-distance (i.e., $d_{(X,\phi)}$ is non-negative, symmetric in its arguments, and satisfies the triangle inequality).
\end{defn}

\indent The following basic example shows that the Kobayashi $\phi$-pseudo-distance need not be a genuine distance function.

\begin{example} \label{ex:Flat-Dist-Degen} Let $(X, g_X, \phi) = (\R^n, g_0, \phi)$, where $\phi \in \Omega^k(\R^n)$ is a constant-coefficient elliptic calibration.  By Example \ref{ex:FlatElliptic}(b), we have $\mathcal{K}_{(\R^n, \phi)}(v) = 0$ for all $v \in T\R^n$.  Consequently, $d_{(\R^n, \phi)} = 0$.
\end{example}

\indent It is natural to ask which calibrated manifolds $(X, g_X, \phi)$ have the property that $d_{(X,\phi)}$ is a genuine distance function; we will discuss this in $\S$\ref{sub:K-phi}.  Note that even when $d_{(X,\phi)}$ is a distance function, it need not coincide with the distance function $\mathrm{dist}_{g_X}$ arising from the Riemannian metric.  For now, we establish some basic properties.

\begin{convention} For the remainder of this work, all calibrated manifolds are assumed to be replete.
\end{convention}

\begin{prop}[Distance-decreasing] \label{Distance-Decreasing-Main} Let $f \colon (X^m, g_X, \alpha) \to (Y^n, g_Y, \beta)$ be a conformally calibrating map.  For all $p,q \in X$, we have
$$d_{(Y,\beta)}(f(p), f(q)) \leq d_{(X,\alpha)}(p,q).$$
\end{prop}

\begin{proof} Fix $p,q \in X$.  Consider the following two sets of piecewise-smooth curve segments:
\begin{align*}
C_X & = \left\{ \widetilde{\gamma} \colon [0,1] \to X \mid \widetilde{\gamma}(0) = p,\, \widetilde{\gamma}(1) = q \right\} \\
C_Y & = \left\{ \gamma \colon [0,1] \to Y \mid \gamma(0) = f(p),\, \gamma(1) = f(q) \right\}\!.
\end{align*}
Notice that $C_Y$ contains the class of curve segments $\{f \circ \widetilde{\gamma} \colon [0,1] \to Y \mid \widetilde{\gamma} \in C_X\}$.  Therefore,
\begin{align*}
d_{(Y,\beta)}(f(p), f(q)) = \inf_{\gamma \in C_Y}  \int_0^1 \mathcal{K}_\beta(\gamma'(t))\,dt  & \leq \inf_{\widetilde{\gamma} \in C_X} \int_0^1 \mathcal{K}_\beta( (f \circ \widetilde{\gamma})'(t) )\,dt  \\
 & = \inf_{\widetilde{\gamma} \in C_X} \int_0^1 \mathcal{K}_\beta( df_{\gamma(t)}(\gamma'(t)) ) \,dt \\
 & \leq \inf_{\widetilde{\gamma} \in C_X} \int_0^1 \mathcal{K}_\alpha( \gamma'(t) )\,dt \tag{Proposition \ref{prop:Dec1}} \\
 & = d_{(X,\alpha)}(p,q).
 \end{align*}
\end{proof}

\begin{cor}[Comparison principle] \label{cor:Comparison-dist} Let $(X, g_X, \phi)$ be a calibrated manifold, and let $U \subset X$ be an open set.  Then $d_{(X, \phi)}(p,q) \leq d_{(U, \phi)}(p,q)$ for all $p,q \in U$.
\end{cor}

\begin{proof} Apply Proposition \ref{Distance-Decreasing-Main} to the inclusion map $\iota \colon (U, g_X|_U, \phi) \hookrightarrow (X, g_X, \phi)$.
\end{proof}

\begin{cor}[Invariance] \label{cor:Invariance-Distance} Let $(X, g_X, \phi)$, $(Y, g_Y, \psi)$ be calibrated manifolds.  If $F \colon X \to Y$ is a Smith equivalence, then for all $p,q \in X$, we have $d_\psi(F(p), F(q)) = d_\phi(p,q)$.  In particular, every Smith automorphism $X \to X$ is an isometry of $d_\phi$.
\end{cor}

\begin{proof} Let $F \colon (X, g_X, \phi) \to (Y, g_X, \psi)$ be a Smith equivalence, and fix $p,q \in X$.  Using Proposition \ref{Distance-Decreasing-Main} twice, and observing that $F^{-1}$ is a Smith equivalence, we have:
$$d_\psi( F(p), F(q) ) \leq d_\phi( p, q ) = d_\phi( F^{-1}( F(p) ), F^{-1}( F(q) ) ) \leq d_\psi( F(p), F(q)).$$
\end{proof}

\indent Next, we provide some examples of Kobayashi pseudo-distances.

\begin{prop}[Kobayashi $\phi$-distances on hyperbolic spaces] \label{prop:HypSpaces-KobHyp} ${}$
\begin{enumerate}[(a)]
\item On real hyperbolic space $(B^n, g_1)$, we have $d_{(B^n, \vol_1)}(p,q) = \mathrm{dist}_1(p,q)$ for all $p,q \in B^n$.
\item On quaternionic hyperbolic space $(\mathbb{HH}^n, g_{\mathbb{HH}^n})$, we have  $d_{(\mathbb{HH}^n, \Psi)}(p,q) = \mathrm{dist}_{\mathbb{HH}^n}(p,q)$ for all $p,q \in \mathbb{HH}^n$.
\end{enumerate}
\end{prop}

\begin{proof} We prove part (a).  Let $p,q \in B^n$, and let $C$ be the collection of all piecewise-smooth curve segments $\gamma \colon [0,1] \to B^n$ with $\gamma(0) = p$ and $\gamma(1) = q$.  Then by Theorem \ref{thm:K-Poincare}, we have
$$d_{(B^n, \vol_1)}(p,q) = \inf_{\gamma \in C} \int_0^1 \mathcal{K}_{(B^n, \vol_1)}(\gamma'(t))\,dt = \inf_{\gamma \in C} \int_0^1 \left| \gamma'(t) \right|_1 dt = \mathrm{dist}_1(p,q).$$
The proof of part (b) is completely analogous, but utilizes Theorem \ref{thm:Quat-Hyp}.
\end{proof}

\begin{cor} \label{cor:Compare-Poincare-Dist} Let $f \colon (B^k, g_1, \vol_1) \to (X, g, \phi)$ be a Smith immersion, where $g_1$ denotes the Poincar\'{e} metric on $B^k$.  For all $p, q \in B^k$, we have $d_{(X,\phi)}(f(p), f(q)) \leq \mathrm{dist}_1(p,q)$.
\end{cor}

\begin{proof} This follows from Propositions \ref{Distance-Decreasing-Main} and \ref{prop:HypSpaces-KobHyp}(a).
\end{proof}

\indent The previous corollary can be strengthened in the following way.

\begin{prop} Let $f \colon (\Sigma^k, g_\Sigma, \vol_\Sigma) \to (X, g_X, \phi)$ be a Smith immersion. 
 Suppose that $(\Sigma, g_\Sigma)$ is a complete Riemannian manifold with constant curvature $-4$ whose universal cover is $B^k$.  Then for all $p,q \in \Sigma$, we have $d_{(X,\phi)}(f(p), f(q)) \leq \mathrm{dist}_\Sigma(p,q)$.
\end{prop} 

\begin{proof} Let $f \colon (\Sigma, g_\Sigma, \vol_\Sigma) \to (X, g_X, \phi)$ be a Smith immersion.  Let $\pi \colon B^k \to \Sigma$ be the universal cover, and note that $\pi^*g_\Sigma$ is isometric to the Poincar\'{e} metric $g_1$ by the Killing-Hopf theorem.  Therefore, $\pi \colon (B^k, g_1, \vol_1) \to (\Sigma, g_\Sigma, \vol_\Sigma)$ is a local Smith equivalence.  Let $p, q \in \Sigma$, and choose $a,b \in B^k$ with $a \in \pi^{-1}(p)$ and $b \in \pi^{-1}(q)$.   \\
\indent Let $F := f \circ \pi \colon (B^k, g_1, \vol_1) \to (X, g_X, \phi)$, so that $F$ is a Smith immersion (Proposition \ref{prop:Invariance}).  Note that $F(a) = f(p)$ and $F(b) = f(q)$.  So, Corollary \ref{cor:Compare-Poincare-Dist} implies
\begin{equation} \label{eq:CoverDistance}
d_\phi(f(p), f(q)) = d_\phi(F(a), F(b)) \leq \mathrm{dist}_1(a,b).
\end{equation}
Now, since $\Sigma$ is complete, there exists a minimizing geodesic $\gamma \colon [0,1] \to \Sigma$ from $p$ to $q$.  Let $\widetilde{\gamma} \colon [0,1] \to B^k$ be a lift of $\gamma$ having $\widetilde{\gamma}(0) = a$, and set $\widetilde{b} = \widetilde{\gamma}(1)$.  Since $\pi$ is a local isometry and $\gamma$ is minimizing, we have $\mathrm{dist}_\Sigma(p,q) \leq \mathrm{dist}_1(a, \widetilde{b}) \leq L(\widetilde{\gamma}) = L(\pi \circ \widetilde{\gamma}) = L(\gamma) = \mathrm{dist}_\Sigma(p,q)$.  Finally, since (\ref{eq:CoverDistance}) is true for all $b \in \pi^{-1}(q)$, we conclude that $d_\phi(f(p), f(q)) \leq \mathrm{dist}_1(a, \widetilde{b}) = \mathrm{dist}_\Sigma(p,q)$.
\end{proof}

\subsection{$K_\phi$-hyperbolicity} \label{sub:K-phi}

\indent \indent We are finally in a position to define $K_\phi$-hyperbolicity.  Let $(X, g_X, \phi)$ be a $\phi$-replete calibrated manifold, and recall that its Kobayashi $\phi$-pseudo-distance is the function $d_\phi \colon X \times X \to [0,\infty]$ given by
\begin{align*}
d_\phi(p,q) & = \inf_\gamma \int_0^1 \mathcal{K}_\phi( \gamma'(t))\,dt
\end{align*}
where the infimum is over all piecewise-smooth paths $\gamma \colon [0,1] \to X$ with $\gamma(0) = p$ and $\gamma(1) = q$. 

\begin{defn} Say $(X, g_X)$ is \emph{$K_\phi$-hyperbolic} if $d_\phi$ is a distance function (i.e., if $p, q \in X$ with $p \neq q$ implies $d_\phi(p,q) > 0$).
\end{defn}

\begin{prop} Let $(X, g_X, \phi)$ be a calibrated manifold, $\phi \in \Omega^k(X)$.
\begin{enumerate}[(a)]
\item Suppose $(X, g_X, \phi)$ is Smith equivalent to $(Y, g_Y, \psi)$.  Then $(X, g_X)$ is $K_\phi$-hyperbolic if and only if $(Y, g_Y)$ is $K_\psi$-hyperbolic.
\item Let $U \subset X$ be an open subset.  If $(X, g_X)$ is $K_\phi$-hyperbolic, then $(U, g_X|_U)$ is $K_\phi$-hyperbolic.
\end{enumerate}
\end{prop}

\begin{proof} Part (a) follows from Corollary \ref{cor:Invariance-Distance}.  Part (b) follows from Corollary \ref{cor:Comparison-dist}.
\end{proof}

\begin{example}[K\"{a}hler calibrations] Let $(X, g_X, \omega)$ be a K\"{a}hler manifold.  Royden \cite{royden2006remarks} proved:
$$X \text{ is } R_\omega\text{-hyperbolic} \ \iff \ X \text{ is }K_\omega\text{-hyperbolic} \ \, (\text{i.e., }X \text{ is Kobayashi hyperbolic}).$$
\end{example}

\noindent We now show that one of these implications holds in general.

\begin{thm} \label{thm:K-implies-Phi} Let $(X, g_X, \phi)$ be a $\phi$-replete calibrated manifold.
\begin{enumerate}[(a)]
\item If $(X, g_X)$ is $R_\phi$-hyperbolic, then $(X, g_X)$ is $K_\phi$-hyperbolic.
\item If $(X, g_X)$ is $K_\phi$-hyperbolic, then $(X, g_X)$ is $\phi$-hyperbolic.
\end{enumerate}
\end{thm}

\begin{proof} ${}$
\begin{enumerate}[(a)]
\item Suppose $(X, g_X)$ is $R_\phi$-hyperbolic, and let $p,q \in X$.  By $R_\phi$-hyperbolicity, there exists a neighborhood $U \subset X$ of $p$ and a constant $c > 0$ such that $\mathcal{K}_{(X,\phi)}(v) \geq c|v|_X$ for all $v \in TU$.  Let $B \subset U$ denote a geodesic ball centered at $p$ of radius $r < \frac{1}{2}\mathrm{dist}_X(p,q)$.  Now, let $\gamma \colon [0,1] \to X$ be a piecewise-smooth path from $p$ to $q$, and let $t_0 \in (0,1)$ be the first time that $\gamma$ crosses the boundary of $B$.  Then
$$\int_0^1 \mathcal{K}_{\phi}(\gamma'(t))\,dt  \geq \int_0^{t_0} \mathcal{K}_\phi(\gamma'(t))\,dt \geq c \int_0^{t_0} |\gamma'(t)|_X\,dt \geq cr.$$
Taking infima over all relevant $\gamma \colon [0,1] \to X$ shows that $d_{(X,\phi)}(p,q) \geq cr > 0$, and thus $(X, g_X)$ is $K_\phi$-hyperbolic. 
\item We prove the contrapositive.  Suppose that $(X,g_X)$ is not $\phi$-hyperbolic, so there exists a non-constant Smith immersion $f \colon (\R^k, g_0, \vol_0) \to (X, g_X, \phi)$.  Since $f$ is not constant, there exist points $p,q \in \R^k$ with $p \neq q$ such that $f(p) \neq f(q)$.  By Proposition \ref{Distance-Decreasing-Main} and Example \ref{ex:Flat-Dist-Degen}, we have $d_{(X,\phi)}( f(p), f(q) ) \leq d_{(\R^k, \vol_0)}(p,q) = 0$.  Thus, $d_{(X,\phi)}$ is not a distance function, so $(X, g_X)$ is not $K_\phi$-hyperbolic.
\end{enumerate}
\end{proof}

 We do not know whether the converse of Theorem \ref{thm:K-implies-Phi}(a) holds or fails in general.  Regarding part (b), we now show that its converse is false without further assumptions.

\begin{prop} Let $\phi \in \Lambda^k(\R^n)^*$ be an inner M\"{o}bius rigid, constant-coefficient, elliptic calibration with $k \geq 2$.  Let $U \subset \R^n$ be the $\phi$-hyperbolic domain constructed in the proof of Theorem \ref{thm:CounterConverse}, and suppose that $(U, g_0)$ is $\phi$-replete.  Then $(U, g_0)$ is not $K_\phi$-hyperbolic.
\end{prop}

\begin{proof} Let $p = (a,0) \in U$ and $q = (b,0) \in U$.  Define sequences of points $p_n, q_n \in U$ for $n \geq 2$ by $p_n = p + \frac{1}{n}u_1$ and $q_n = q + \frac{1}{n}u_1$.  Observe that
$$d_\phi(p,q) \leq d_\phi( p, p_n) + d_\phi( p_n, q_n) + d_\phi(q_n, q).$$
We will prove that $d_\phi(p, p_n) \to 0$ and $d_\phi(q, q_n) \to 0$, and then that $d_\phi(p_n, q_n) \to 0$. \\
\indent First, consider the curves $\gamma_n \colon [0,1] \to U$ given by $\gamma_n(t) = p + \frac{t}{n}u_1$.   Note that $\gamma_n$ is well-defined (i.e., the image of $\gamma_n$ lies in $U$).  Note also that $\gamma_n(0) = p$ and $\gamma_n(1) = q$, and that $\gamma_n'(t) = \frac{1}{n}u_1$.  Therefore,
\begin{align*}
d_\phi(p, p_n) & \leq \int_0^1 \mathcal{K}_\phi(\gamma'(t))\,dt = \frac{1}{n} \int_0^1 \mathcal{K}_\phi(u_1)\,dt \to 0 \text{ as } n \to \infty.
\end{align*}
An analogous argument shows that $d_\phi(q, q_n) \to 0$. \\
\indent Next, for $n \geq 2$, consider the linear Smith immersions $f_n \colon (B^k, g_1, \vol_1) \to (U, g_0, \phi)$ defined by $f_n(z) = n(z,0) + \frac{1}{n}u_1$.  Note that the image of $f_n$ is the $k$-ball $\frac{1}{n}u_1 + \{(x,0) \colon |x| < n\}$, and that $f_n(\frac{1}{n}a) = p_n$ and $f_n(\frac{1}{n}b) = q_n$.  Using Corollary \ref{cor:Compare-Poincare-Dist} and Proposition \ref{prop:Poincare-Facts}, we have
\begin{align*}
d_\phi( p_n, q_n) = d_\phi\!\left( f_n\!\left(\frac{1}{n}a\right), f_n\!\left(\frac{1}{n}b\right) \right) & \leq \mathrm{dist}_1\!\left( \frac{1}{n}a, \frac{1}{n}b\right) \\
& = \mathrm{arcsinh} \sqrt{ \frac{ \frac{1}{n^2} \left| a-b \right|^2 }{ (1 - \frac{1}{n^2}|a|^2) (1 - \frac{1}{n^2}|b|^2) } } \\
& \to 0 \text{ as } n \to \infty.
\end{align*}
This proves that $U$ is not $K_\phi$-hyperbolic.
\end{proof}

\section{Questions}

\indent We conclude this work by raising the following questions.
\begin{enumerate}
\item Which calibrated manifolds $(X, g_X, \phi)$ are $\phi$-replete?  Proving the $\phi$-repleteness of $(X, g_X)$ appears to require the existence of a large supply of conformally flat, immersed $\phi$-calibrated $k$-disks in $X$.  Given the difficulty of constructing calibrated submanifolds, this question appears to be non-trivial.
\item When $(X, g, \omega)$ is a K\"{a}hler manifold, Brody's theorem \cite{brody1978compact} asserts that if $X$ is compact, then $K_\omega$-hyperbolicity is equivalent to $\omega$-hyperbolicity.  Are there other calibrated geometries $(X, g, \phi)$ for which Brody's theorem holds?  That is, if $X$ is compact, is $K_\phi$-hyperbolicity equivalent to $\phi$-hyperbolicity?  In the Riemannian geometric study of hyperbolicity via conformal harmonic disks, a partial analog of Brody's theorem is established in \cite[Theorem 6.7]{gaussier2024kobayashi}.
\item When $(X, g, \omega)$ is a compact K\"{a}hler manifold, it is known that the $K_\omega$-hyperbolicity of $X$ implies that the biholomorphism group of $X$ (and hence also the Smith automorphism group $\mathrm{SmAut}(X, g, \omega)$) is finite.  Are there other calibrated geometries for which compactness and $K_\phi$-hyperbolicity together imply the finiteness of $\mathrm{SmAut}(X, g, \phi)$?
\item In view of Theorem \ref{cor:NegativeCurvature}, it would be interesting to have examples of calibrated manifolds $(X, g_X, \phi)$ having negative $\phi$-sectional curvature (without having negative sectional curvature).
\item Arguably the simplest example of a hyperbolic manifold (for essentially any notion of ``hyperbolicity") is hyperbolic space $\mathbb{H}^n$, which we have viewed as the Poincar\'{e} ball $(B^n, g_1)$.  As such, it would be interesting to have a large supply of calibrations on $\mathbb{H}^n$.
\end{enumerate}

\bibliographystyle{plain}
\bibliography{Hyperbolicity-Ref}

\Addresses

\end{document}